\documentclass[12p]{article}
\usepackage{amscd}
\usepackage{amsthm}
\usepackage {geometry}
\usepackage {fancyhdr}
\usepackage {amsmath,amsthm,amssymb}
\usepackage {graphicx}
\usepackage {hyperref}
\usepackage{amsfonts}
\usepackage{tabularx}
\usepackage{enumerate}
\usepackage{undertilde}
\usepackage[normalem]{ulem}
\usepackage[usenames,dvipsnames,svgnames,table]{xcolor}
\usepackage{relsize}
 \usepackage[titletoc,title]{appendix}

\newtheorem{theorem}{Theorem}[section]
\theoremstyle{definition}
\newtheorem {definition}[theorem]{Definition}

\newtheorem{lemma}[theorem]{Lemma}
\newtheorem{proposition}[theorem]{Proposition}

\theoremstyle{remark}\newtheorem{remark}[theorem]{Remark}

\newcommand{\E}{\mathbb{E}}

\newcommand{\Z}{\mathbb{Z}}
\newcommand{\R}{\mathbb{R}}
\renewcommand{\P}{\mathbb{P}}

 \DeclareMathOperator{\Var}{Var}
\newcommand{\old}[1]{{}}

\newcommand{\Addresses}{{
  \bigskip
  \footnotesize

  Yiting Li, \textsc{Department of Mathematics, Brandeis University,
415 Soutrh Street, Waltham, MA 02453, USA}\par\nopagebreak
  \textit{E-mail address}, Yiting Li: \texttt{yitingli@brandeis.edu}

  \medskip

  Xin Sun, \textsc{Department of Mathematics, Massachusetts Institute of Technology,
    77 Massachusetts Avenue, Cambridge, MA 02139-4307, USA}\par\nopagebreak
  \textit{E-mail address}, Xin Sun: \texttt{xinsun89@math.mit.edu}

}}

\def\be{\begin{equation}}
\def\ee{\end{equation}}

\begin{document}

\title{On Fluctuations for Random Band Toeplitz Matrices}

\date{November 8, 2013}
\author{ Yiting Li$^1$  and  Xin Sun$^2$ } 

\date{}
\maketitle

\begin{abstract}
In this paper we study two one-parameter families of random band
Toeplitz matrices:
\[
A_n(t)=\frac{1}{\sqrt{b_n}}\Big(a_{i-j}\delta_{|i-j|\le[b_nt]}\Big)_{i,j=1}^n
\quad\text{and}\quad
B_n(t)=\frac{1}{\sqrt{b_n}}\Big(a_{i-j}(t)\delta_{|i-j|\le
b_n}\Big)_{i,j=1}^n
\]
where
\begin{enumerate}
\item $a_0=0$, $\{a_1,a_2,\ldots\}$ in $A_n(t)$ are independent random variables
and $a_{-i}=a_i$;
\item $a_0(t)=0$, $\{a_1(t),a_2(t),\ldots\}$ in
$B_n(t)$ are independent copies of the standard Brownian motion at
time $t$ and $a_{-i}(t)=a_i(t)$.
\end{enumerate} 
As $t$ varies, the empirical measures $\mu(A_n(t))$ and
$\mu(B_n(t))$ are measure valued stochastic processes. The purpose
of this paper is to study the fluctuations of $\mu(A_n(t))$ and
$\mu(B_n(t))$ as $n$ goes to $\infty$. Given a monomial $f(x)=x^p$
with $p\ge2$, the corresponding rescaled fluctuations of
$\mu(A_n(t))$ and $\mu(B_n(t))$ are
\begin{align}
\sqrt{b_n}\Big(\int f(x)d\mu(A_n(t))-\E[\int f(x)d\mu(A_n(t))]\Big)=\frac{\sqrt{b_n}}{n}\Big(\text{tr}(A_n(t)^p)-\E[\text{tr}(A_n(t)^p)]\Big),\label{1}\\
\sqrt{b_n}\Big(\int f(x)d\mu(B_n(t))-\E[\int
f(x)d\mu(B_n(t))]\Big)=\frac{\sqrt{b_n}}{n}\Big(\text{tr}(B_n(t)^p)-\E[\text{tr}(B_n(t)^p)]\Big)\label{2}
\end{align}
respectively. We will prove that \eqref{1} and \eqref{2} converge to
centered Gaussian families $\{Z_p(t)\}$ and $\{W_p(t)\}$
respectively. The covariance structure $\E[Z_p(t_1)Z_q(t_2)]$ and
$\E[W_p(t_1)W_q(t_2)]$ are obtained for all $p,q\ge 2,t_1,t_2\ge 0,$
and are both homogeneous polynomials of $t_1$ and $t_2$ for fixed
$p,q$. In particular, $Z_2(t)$ is the Brownian motion and $Z_3(t)$
is the same as $W_2(t)$ up to a constant.

The main method of this paper is the moment method.
\end{abstract}

\section{Introduction}\label{sec: introduction}


In random matrix theory, one fundamental object is the empirical
distribution of eigenvalues. For an $n$ by $n$ real symmetric random
matrix $T$, we use $\mu(T)$ to denote its empirical distribution:
\[
\mu(T)=\frac{1}{n}\sum\limits_{i=1}^n\delta(x-\lambda_i)
\]
where $\lambda_1,\ldots,\lambda_n$ are the $n$ real eigenvalues of
$T$. The asymptotic behavior of the empirical distribution has
played an essential role in random matrix theory since Wigner's
semicircle law (see \cite{Wigner1} and \cite{Wigner2}).  \cite{AGZ}
and \cite{Mehta2} are standard references of the various results in
the fifty years after that.


 In recent years, random matrices with
certain linear structures are well studied. One important example is
the random Toeplitz matrice.
Bryc, Dembo and Jiang \cite{bdj}
proved the existence of the limit of empirical distribution of
symmetric Toeplitz matrix. Hammond and Miller \cite{hm} also proved
this existence independently. Liu and Wang \cite{lw} proved the
existence of the limit of empirical distribution of symmetric
Toeplitz band matrix. Some other interesting results about random
Toeplitz matrices can be found in \cite{bb,bcg,bm,kargin,mms,SV}.

The random band matrices have connections with the theory of quantum
chaos, see \cite{cmi} and \cite{fm}. For random band matrices of
Wigner type, the limit of the empirical distribution was studied in
\cite{bmp,mpk}. It is believed that the local statistics has a
transition from Poisson statistics to GUE or GOE statistics when the
bandwidth crosses $\sqrt n$ (see \cite{fm}). For recent process on
local stabilities of Wigner type band matrices see \cite{EKY,EKYY}
and the reference therein.

Fluctuations of random matrices is a classical topics in this field
now. Some important literature about fluctuations of eigenvalues
includes \cite{BS,DEvans,DS,DE,johansson,jonsson,popescu,SS}. The
first paper concerning fluctuation of random band matrices is
\cite{AZ}. For recent development, see \cite{jss,ls}. The
fluctuation of random Toeplitz matrices was first studied by
Chatterjee \cite{chatterjee} in the case where the matrix entries
are normally distributed. In \cite{lsw} Liu, Sun and Wang derived a
central limit theorem for the fluctuation of random Toeplitz
matrices with general entries.

In this paper we study the fluctuations of two models of random band
Toeplitz matrices each of which contains a nonnegative parameter
$t$.


\subsection{Fluctuations of linear statistics for matrix model with bandwidth proportional to $t$}\label{sec: fluctuation result}
Our first model is the fluctuation of eigenvalues of random band
Toeplitz matrix with bandwidth proportional to $t$.  The result of
this subsection is inspired by \cite{lsw}.

Let $ a_0 =0,\{ a_i | i\in \Z\backslash \{0 \} \}$ be real random
variables such that
\begin{align}\label{eq:condition on entries for CLT}
&\E[ a_i ]=0, \Var[a_i ] =1, E[|a_i|^4]=\kappa,\nonumber\\
&\E[|a_i|^k]<C_k<\infty \textrm{ for } k>2,\nonumber\\
&a_1,a_2,\cdots\textrm{ are independent, }\nonumber\\
&a_{-i}=a_i,\;\forall i\in \Z,
\end{align}
where $ \kappa\ge1$. Consider the $n\times n$ band random matrix
\begin{align*}
A_n(t)=\frac{1}{\sqrt{b_n}}\Big(a_{i-j}\delta_{|i-j|\le[b_nt]}\Big)_{i,j=1}^n
\end{align*} where $b_n\le n$, $b_n\to\infty$ as $n\to\infty$ while $\lim\limits_{n\to\infty}\frac{b_n}{n}=b\in[0,1]$.

 In \cite{lsw}, Liu, Sun and Wang studied the fluctuation of
moments of random band Toeplitz matrices and got the following
theorem .
\begin{theorem}\label{thm: fluctuation thm of Liu,Sun,Wang}
Suppose $A_n(t)$ is defined as above. Then for any $p\ge2$, the
fluctuation of the $p$-th moment of $A(1)$ (with rescaling)
converges weakly to a Gaussian distribution:
\[
\frac{\sqrt{b_n}}{n}(\text{tr}A(1)^p-\E[\text{tr}A(1)^p])\to\text{N}(0,\sigma_p^2).
\]
\end{theorem}
The variances $\sigma_p^2$ will be given in Remark \ref{remark:
covairance of Liu Sun Wang's model}. In this paper we consider the
time-dependent fluctuation
\begin{equation}\label{eq:omega p}
\omega_p(t):=\sqrt{b_n}\Big(\int x^pd\mu(A_n(t))-\E[\int
x^pd\mu(A_n(t))]\Big)=\frac{\sqrt{b_n}}{n}(\text{tr}A_n(t)^p-\E[\text{tr}A_n(t)^p])
\end{equation}
as a stochastic process with parameter $t\ge0$.

For natural numbers $p$, $q$ and $k\le\min\{p,q\}$, define
$R_1(p,q,k)={p\choose k}{q\choose k}k!(p-k-1)!!(q-k-1)!!$,
$R_2(p,q)=\frac{p\,q}{4}(p-1)!!(q-1)!!$. We make the convention that
$(-1)!!=1$ and that $(0,1/b)$ denotes $(0,\infty)$ when $b$ equals
$0$. Our first main result is the following theorem.

\begin{theorem}\label{thm:multi time fluctuation}

When $n\to\infty$, $\{\omega_p(t)|p\ge2\}$ jointly converge to a
family of Gaussian processes $\{Z_p(t)|p\ge2\}$ in the following
sense. Suppose
\begin{enumerate}
\item $p_1,\ldots,p_r$ are natural numbers no less than 2;
\item $t_1<\cdots<t_r$ are numbers in $(0,1/b)$;
\item $\{a_1,\ldots,a_r\}\subset\mathbb{R}$,
\end{enumerate} then
\[
\lim\limits_{n\to\infty}\P(\omega_{p_1}(t_1)\le
a_1,\ldots,\omega_{p_r}(t_r)\le a_r)=\P(Z_{p_1}(t_1)\le
a_1,\ldots,Z_{p_r}(t_r)\le a_r).
\]
The expectation of $Z_p(t)$ is 0 for all $t\ge0,\,p\ge2$. The
covariance structure of $\{Z_p(t)|p\ge2\}$ will be given in Section
\ref{sec: fluctuation}. In particular, when $b=0$, for $p\ge2$,
$q\ge2$ and $0<t_1\le t_2$,
\begin{align}\label{eq: covariance}
&\E[Z_{p}(t_1)\,Z_{q}(t_2)]\nonumber\\
=&\begin{cases}\sum\limits_{k=3,5,\ldots,\min\{p,q\}}R_1(p,q,k)t_1^{\frac{p+k}{2}-1}t_2^{\frac{q-k}{2}}2^{\frac{p+q}{2}}&\text{if }p,q\text{ are both odd}\\
\sum\limits_{k=4,6,\ldots,\min\{p,q\}}R_1(p,q,k)t_1^{\frac{p+k}{2}-1}t_2^{\frac{q-k}{2}}2^{\frac{p+q}{2}}+(\kappa-1)R_2(p,q)t_1^{\frac{p}{2}}t_2^{\frac{q}{2}-1}2^{\frac{p+q}{2}}&\text{if
}p,q\text{ are both even}\\
0&\text{otherwise}
\end{cases}
\end{align}
which is a homogeneous polynomial of $t_1$ and $t_2$. From the
covariance structure, if $b=0$, then $Z_2(t)$ is the Brownian motion
and $\E[Z_3(t_1)Z_3(t_2)]=48(t_1\wedge t_2)^2$.
\end{theorem}

\begin{remark}\label{remark: results for polynomial for the first
model}

From Theorem \ref{thm:multi time fluctuation} we know that
$$\omega_Q(t)=\frac{\sqrt{b_n}}{n}\Big(\text{tr}Q(A_n(t))-\E[\text{tr}Q(A_n(t))]\Big)$$
converges to a Gaussian process where
$Q(x)=\sum\limits_{j=2}^mq_jx^j$. The correlation structure of the
limit process can be computed via the covariance structure of
$\{Z_p(t)|p\ge2\}$. When $b=0$, the correlation structure of the
limit process can be computed via Equation \eqref{eq: covariance}.
\end{remark}

\subsection{Fluctuations of linear statistics for matrix model with Brownian motion
entries}\label{sec: fluctuation for Brownian entries} Our second
model is the fluctuation of the eigenvalues of the random band
Toeplitz matrix with Brownian motion entries.

One important matrix model with Brownian motion entries is the Dyson
Brownian motion $H_n(t)=(h_{ij}(t))_{i,j=1}^n$ where
$\{h_{ij}(t)|i\ge j\}$ are independent Brownian motions. The Dyson
Brownian motion was first studied by Dyson who derived a stochastic
differential system for eigenvalues and eigenvectors of $H_n(t)$.
Since then the Dyson Brownian motion has been well studied and
becomes a very useful tool in random matrix theory. See, for
example, \cite{AGZ}.

Our model also has Brownian motions as matrix entries, but with a
Toeplitz structure.

Let $a_0(t)\equiv0$ and $a_1(t),a_2(t),\ldots$ be independent
standard Brownian motions with time $t$. Set $a_{-i}(t)=a_i(t)$.
Suppose $b_n\to\infty$ as $n\to\infty$ and $b_n=o(n)$. Consider the
$n\times n$ random band matrix
\begin{align}\label{model: A_n(t) for Brownian entries}
B_n(t)=\frac{1}{\sqrt{b_n}}\Big(a_{i-j}(t)\delta_{|i-j|\le
b_n}\Big)_{i,j=1}^n.
\end{align}

For any integer $p\ge2$, define the time dependent fluctuations in
the same way as (\ref{eq:omega p}), i.e.,

\begin{equation}\label{eq: omega for brownian entries}
\omega_p(t):=\frac{\sqrt{b_n}}{n}\Big(\text{tr}(B_n(t)^p)-\E[\text{tr}(B_n(t)^p)]\Big).
\end{equation}

For natural numbers $p$, $q$ and $k\le\min\{p,q\}$, define
$R_3(p,r)=(p+r-1)!!-(p-1)!!(r-1)!!$,
$R_4(p,r)=(p+r-1)!!-p\,r(p-2)!!(r-2)!!$. We make the convention that
$(-1)!!=1$. Our second main result is the following theorem.

\begin{theorem}\label{thm: fluctuation for Brownian entries}

Assume $b_n\to\infty$ as $n\to\infty$ and $b_n=o(n)$. When
$n\to\infty$, $\{\omega_p(t)|p\ge2\}$ jointly converge to a family
of Gaussian processes $\{W_p(t)|p\ge2\}$ in the following sense.
Suppose $p_1\ge2$,\ldots,$p_r\ge2$, $0<t_1<\cdots<t_r$ and
$\{a_1,\ldots,a_r\}\subset\mathbb{R}$, then
\[
\lim\limits_{n\to\infty}\P(\omega_{p_1}(t_1)\le
a_1,\ldots,\omega_{p_r}(t_r)\le a_r)=\P(W_{p_1}(t_1)\le
a_1,\ldots,W_{p_r}(t_r)\le a_r).
\]
The expectation of $W_p(t)$ is 0 for all $t\ge0,\,p\ge2$. For
$p\ge2$, $q\ge2$ and $0<t_1\le t_2$,
\begin{align}\label{eq: covariance for the second model}
\lim\limits_{n\to\infty}\E[\omega_p(t_1)\omega_q(t_2)]=&\begin{cases}\sum\limits_{r=2,4,\ldots,q}{q\choose
r}t_1^{\frac{p+r}{2}}R_3(p,r)(q-r-1)!!(t_2-t_1)^{\frac{q-r}{2}}2^{\frac{p+q}{2}}&\text{if
$p$ and $q$ are both even}\\\sum\limits_{r=3,5,\ldots,q}{q\choose
r}t_1^{\frac{p+r}{2}}R_4(p,r)(q-r-1)!!(t_2-t_1)^{\frac{q-r}{2}}2^{\frac{p+q}{2}}&\text{if
$p$ and $q$ are both odd}\\0&\text{otherwise}\end{cases}
\end{align}
which is a homogeneous polynomial of $t_1$ and $t_2$. Directly from
the covariance structure, if $p=2$, then
$\E[W_2(t_1)W_2(t_2)]=8(t_1\wedge t_2)^2$; if $p=3$, then
$\E[W_3(t_1)W_3(t_2)]=48(t_1\wedge t_2)^3$. In particular, the limit
process $Z_3(t)$ in Theorem \ref{thm:multi time fluctuation} equals
$\sqrt6W_2(t)$ when $b=0$.
\end{theorem}

\begin{remark}
Similarly as in Remark \ref{remark: results for polynomial for the
first model}, from Theorem \ref{thm: fluctuation for Brownian
entries} we know that for $Q(x)=\sum\limits_{j=2}^pq_jx^j$,
$$\omega_Q(t)=\frac{\sqrt{b_n}}{n}\Big(\text{tr}Q(B_n(t))-\E[\text{tr}Q(B_n(t))]\Big)$$
converge to a Gaussian process whose correlation structure can be
computed via Equation \eqref{eq: covariance for the second model}.
\end{remark}

\begin{remark}
 It will be clear from our proof that Theorem \ref{thm: fluctuation for Brownian entries} can be
generalized to a stronger version. If $a_1(t),a_2(t),\ldots$ are
independent copies of $S_t$ which is a stochastic process such that
\begin{enumerate}
\item
$S_t$ has independent increments;
\item for all $t\ge 0$,
$S_t$ is centered with finite moments,
\end{enumerate}
then $\{\omega_p(t)|p\ge2\}$ also jointly converge to a family of
centered Gaussian process whose covariance structure can be obtained
in the same way as in the proof of Theorem \ref{thm: fluctuation for
Brownian entries}.
\end{remark}
\subsection{Outline}\label{sec: outline}
Section \ref{sec:preliminaries} and Appendix \ref{appendix} provide
some preliminary concepts and results. Theorem \ref{thm:multi time
fluctuation} is proved in Section \ref{sec: fluctuation}. We first
study the asymptotic covariance structure of $\omega_p(t)$ by the
moment method in Section \ref{sec: covariance for fluctuation} and
prove Theorem \ref{thm:multi time fluctuation} by showing that
$(\omega_{p_1}(t_1),\ldots,\omega_{p_r}(t_r))$ converges weakly to a
centered multivariate Gaussian distribution
in Section \ref{sec:Gaussian fluctuation proof}. 

Theorem \ref{thm: fluctuation for Brownian entries} is proved in
Section \ref{sec: last section: fluctuation for Brownian entries}.
 We first prove an ``asymptotic commutativity" lemma for random band Toeplitz matrices with
slowly growing bandwidth in Section \ref{sec: asymptotic
commutatuvity}. Then we study the asymptotic covariance structure of
$\omega_p(t)$ by the moment method in Section \ref{sec: covariance
for Brownina entries}. In Section \ref{multi-point fluctuation for
Brownina entries}, we complete the proof  by showing that
$(\omega_{p_1}(t_1),\ldots,\omega_{p_r}(t_r))$ converges weakly to a
centered multivariate Gaussian distribution.

\section{Preliminaries}\label{sec:preliminaries}
In this section we give some notations and facts that we use
throughout the paper.

\subsection{Trace formula}\label{sec: trace formula}
In \cite{lsw} Lemma 3.2, Liu, Sun and Wang proved a useful trace
formula of the product of band Toeplitz matrices. The trace lemma
requires the matrices to have the same bandwidth, but their proof
also applies for the case that the matrices have different
bandwidths. Therefore the trace lemma can be generalized to the
following version.

\begin{lemma}\label{lemma:trace formula}
Consider band Toeplitz matrices $T_{l,n}=(a_{l,i-j}\delta_{i-j\le
b_n^{(l)}})_{i,j=1}^{n}$ where $a_{l,-n+1}, \cdots,a_{l,n-1}$ are
 real numbers, $b_n^{(l)}$ is the bandwidth of $T_{l,n}$ and $l=1,\ldots,p$. We have the trace formula
 \begin{equation} \label{basic:lem1}
 \mathrm{tr}( T_{1,n}\cdots T_{p,n})=\sum_{i=1}^{n}\,\sum_{\mathbf{J}}
a_{\mathbf{J}}\,I(i,\mathbf{J}) \
\large{\delta}(\sum\limits_{l=1}^{p}j_{l}).
 \end{equation}
  Here
$\mathbf{J}=(j_{1},\ldots,j_{p})$,
$a_{\mathbf{J}}=\prod\limits_{l=1}^pa_{l,j_{l}}$, $\delta$ is the
Dirac function and the summation $\sum\limits_{\mathbf{J}}$ runs
over all possibilities that $\mathbf{J} \in
\{-b_{n}^{(1)},\ldots,b_{n}^{(1)}\}\times\ldots\times\{-b_{n}^{(p)},\ldots,b_{n}^{(p)}\}$.
\end{lemma}

\subsection{Partition}\label{sec:partition}
In this section we define various types of partitions. We suggest
readers to skip these definitions for a second and refer to them
when they are needed in the following sections.\\

Let $[n]=\{1,2,\ldots,n\}$. We call $\pi=\{V_1,\ldots,V_r\}$ a
partition of $[n]$ if $V_1,\ldots V_r$ are pairwise disjoint,
nonempty subsets of $[n]$ such that $[n]=V_1\cup\cdots\cup V_r$. For
$\forall i\in[n]$, define $\pi(i)=j$ if $i\in V_j$.

We call $\pi$ a pair partition of $[2k]$ if it's a partition of
$[2k]$ and each element of $\pi$ contains two elements of $[2k]$. So
a pair partition must have the form
$\pi=\{\{a_1,b_1\},\ldots,\{a_k,b_k\}\}$. For such a $\pi$ we write
$a_u\sim_\pi b_u$ ($1\le u\le k$). Since such a $\pi$ can be seen as
a permutation: $(a_1,b_1) \cdots (a_k,b_k)$, we define $g(\pi)$ to
be the number of orbits of the permutation $\gamma_0\circ\pi$ where
$\gamma_0=(1,2,\ldots,2k-1,2k)$ is the canonical cycle. We denote by
$\mathcal{P}_2(2k)$ the set of pair partitions of $[2k]$.

Suppose $p$ and $q$ are natural numbers and $p+q$ is even. Suppose
$$\pi=\{\{a_1,b_1\},\ldots,\{a_{(p+q)/2},b_{(p+q)/2}\}\}\in\mathcal{P}_2(p+q).$$
If a block $\{a_i,b_i\}$ of $\pi$ has one element in
$\{1,\ldots,p\}$ and one element in $\{p+1,\ldots,p+q\}$, then we
call $\{a_i,b_i\}$ a cross of $\pi$.

Suppose $p$ and $q$ are natural numbers and $p+q$ is even. We define
$\mathcal{P}_2(p,q)$ to be the subset of $\mathcal{P}_2(p+q)$ such
that each $\pi$ in $\mathcal{P}_2(p,q)$ has at least one cross,
i.e., there are $i\sim_\pi j$ such that $i\le p<j$. If $p+q$ is odd,
then $\mathcal{P}_2(p,q)$ is defined to be $\emptyset$.

Suppose $p$ and $q$ are natural numbers and $p+q$ is even. Define
$\mathcal{\tilde P}_2(p,q)$ to be the subset of $\mathcal{P}_2(p,q)$
consisting of permutations with at least three crosses. If $p+q$ is
odd, then $\mathcal{\tilde P}_2(p,q)$ is defined to be $\emptyset$.

Suppose $p$, $q$ are both even numbers. We use
$\mathcal{P}_{2,4}(p,q)$ to denote the set of partitions of $p+q$
such that each $\pi=\{V_1,\ldots,V_r\}$ in $\mathcal{P}_{2,4}(p,q)$
satisfies:\\
(i) $\exists i$ such that $V_i$ has 4 elements while other $V_j$ has
2 elements each;\\
(ii) if $j\ne i$, then $V_j\subset\{1,\ldots,p\}$ or
$V_j\subset\{p+1,\ldots,p+q\}$;\\
(iii) two elements of $V_i$ come from $\{1,\ldots,p\}$ and the other
two come from $\{p+1,\ldots,p+q\}$.\\
If $p$ and $q$ are not both even, then we define
$\mathcal{P}_{2,4}(p,q)$ to be empty.

\subsection{Balanced vector and cluster}\label{sec: blanced vector and cluster}

\begin{definition}
A vector $\mathbf{J}=(j_1,\ldots,j_k)$ is called a balanced vector
if the sum of its components is 0:
\[
j_1+\cdots+j_k=0.
\]
\end{definition}

\begin{definition}
For a vector $\mathbf{v}=(a_1,\ldots,a_k)$, set
$S_{\mathbf{v}}=\{|a_1|,\ldots,|a_k|\}$ which is a set of
non-negative numbers with multiplicity. Two vectors
$\mathbf{v_1},\mathbf{v_2}$ are called correlated if
$S_{\mathbf{v_1}}$ and $S_{\mathbf{v_2}}$ have at least one common
element. If $a_i$ is a component of $\mathbf{v_1}$ and $|a_i|\in
S_{\mathbf{v_1}}\cap S_{\mathbf{v_2}}$, then $a_i$ is called a joint
point of $\mathbf{v_1}$ and $\mathbf{v_2}$.
\end{definition}

\begin{definition}
Given a set of vectors $\{\mathbf{J_1},\ldots,\mathbf{J_r}\}$.
Suppose the components of each vector are real numbers. A subset
$\mathbf{J_{i_1}},\ldots,\mathbf{J_{i_s}}$ is called a cluster if
\begin{enumerate}
\item for any pair $\mathbf{J_{i_u}}$,
$\mathbf{J_{i_v}}$ from the subset one can find a chain of vectors,
also from the subset, which stars with $\mathbf{J_{i_u}}$, ends with
$\mathbf{J_{i_v}}$, such that any two neighboring vectors in the
chain are correlated;
\item the subset cannot be enlarged with the preservation of condition 1.
\end{enumerate}
We call $s$ the length of this cluster.
\end{definition}
The following lemma tells us that the number of clusters with length
longer than 2 is very small. It was stated and proved in \cite{lsw}.
\begin{lemma}\label{lemma:cluster estimate}
Suppose $l>2$ and $\lim\limits_{n\to\infty}b_n=+\infty$. Set
$\mathcal{B}_{n,p}$ to be
$$\{(j_1,\ldots,j_p)\in\{\pm1,\ldots,\pm b_n\}^p|j_1+\cdots+j_p=0\}.$$
Let $\mathcal{B}_p$ be a subset of
$\mathcal{B}_{n,p_1}\times\cdots\times\mathcal{B}_{n,p_l}$ such that
$(\mathbf{J_1},\ldots,\mathbf{J_l})\in\mathcal{B}_p$ if and only if:\\
1) each element of $\cup_{k=1}^lS_{\mathbf{J_k}}$ has at least multiplicity two;\\
2) $\mathbf{J_1}$,\ldots,$\mathbf{J_l}$ makes a cluster.\\
Then we have that
\[
|\mathcal{B}_p|=o(b_n^{\frac{p_1+\cdots+p_l-l}{2}})
\]
where $|\mathcal{B}_p|$ is the cardinality of $\mathcal{B}_p$.
\end{lemma}

\old{
\subsection{Some preliminary  integrals}\label{sec:integral for covariance}
To prove  Theorem \ref{thm:multi time fluctuation} and Theorem
\ref{thm: fluctuation for Brownian entries}, we need to define four
functions $\tilde f_I^+(\pi,t_1,t_2)$, $\tilde f_I^-(\pi,t_1,t_2)$,
$\tilde f_{II}^+(\pi,t_1,t_2)$ and $\tilde f_{II}^-(\pi,t_1,t_2)$,
but we suggest readers to skip these definitions now and refer to
them when they are mentioned in the proof of Theorem \ref{thm:multi
time fluctuation} and Theorem \ref{thm: fluctuation for Brownian
entries}.

Suppose $b\in[0,1]$ and $t_1\le t_2$ are numbers in $(0,1/b)$. (We
define $(0,1/0)$ to be $(0,\infty)$.)

Suppose $p+q$ is even. For
$\pi=\{V_1,\ldots,V_{(p+q)/2}\}\in\mathcal{P}_2(p,q)$, set
\begin{align*}
\epsilon_{\pi}(i)=\begin{cases}1&\text{if $i$ is the smallest number
of }\pi^{-1}(\pi(i))\\-1&\text{otherwise}\end{cases}
\end{align*}
where $\pi(i)$ is defined in Section \ref{sec:partition}.

We construct a relation between two groups of unknowns
$y_1,\ldots,y_{p+q}$ and $x_1,\ldots,x_{\frac{p+q}{2}}$ as
$$\epsilon_{\pi}(i)y_i=\epsilon_{\pi}(j)y_j=x_{\pi(i)}$$ whenever
$i\sim_{\pi}j$. Without lost of generality suppose
$V_i\cap\{1,\ldots,p\}\ne\emptyset$ if and only if $i\le s$. For
$x_0,y_0\in[0,1]$ and $x_1,\ldots,x_{\frac{p+q}{2}}\in[-1,1]$, we
define
\begin{align*}
\tilde f_I^-(\pi,t_1,t_2)=&\int_{[0,1]^2}dy_0dx_0\int_{[-t_1,t_1]^s}dx_1\cdots dx_s\int_{[-t_2,t_2]^{-s+(p+q)/2}}dx_{s+1}\cdots dx_{(p+q)/2}\\
\times&\delta(\sum\limits_{i=1}^py_i)\prod_{j=1}^pI_{[0,1]}(x_0+b\sum\limits_{i=1}^jy_i)\prod_{j'=p+1}^{p+q}I_{[0,1]}(y_0+b\sum\limits_{i=p+1}^{j'}y_i)
\end{align*}
and
\begin{align*}
\tilde f_I^+(\pi,t_1,t_2)=&\int_{[0,1]^2}dy_0dx_0\int_{[-t_1,t_1]^s}dx_1\cdots dx_s\int_{[-t_2,t_2]^{-s+(p+q)/2}}dx_{s+1}\cdots dx_{(p+q)/2}\\
\times&\delta(\sum\limits_{i=1}^py_i)\prod_{j=1}^pI_{[0,1]}(x_0+b\sum\limits_{i=1}^jy_i)\prod_{j'=p+1}^{p+q}I_{[0,1]}(y_0-b\sum\limits_{i=p+1}^{j'}y_i)
\end{align*}
where $\delta$ is the Dirac function and $I_{[0,1]}$ is the
indicator function.

For
$\pi=\{V_1,\ldots,V_{\frac{p+q}{2}-1}\}\in\mathcal{P}_{2,4}(p,q)$
(denoting the block with four elements by $V_i$), we set for
$\pi(k)\ne i$
\begin{align*}
\tau_{\pi}(k)=\begin{cases}1& \text{if $k$ is the smallest number of
}\pi^{-1}(\pi(k));\\-1&\text{otherwise}.\end{cases}
\end{align*}
while for $\pi(k)=i$
\begin{align*}
\tau_{\pi}(k)=\begin{cases}1& \text{if $k$ is the smallest or
largest number of
}\pi^{-1}(\pi(k));\\-1&\text{otherwise}.\end{cases}
\end{align*}
For such a $\pi$, we define a relation between two groups of
unknowns $y_1,\ldots,y_{p+q}$ and $x_1,\ldots,x_{\frac{p+q}{2}-1}$
as following: $$\tau_{\pi}(u)y_u=\tau_{\pi}(v)y_v=x_{\pi(u)}$$
whenever $u\sim_{\pi}v$.
For $x_0,y_0\in[0,1]$ and $x_1,\ldots,x_{\frac{p+q}{2}-1}\in[-1,1]$,
we define
\begin{align*}
\tilde f_{II}^-(\pi,t_1,t_2)=&\int_{[0,1]^2}dy_0dx_0\int_{[-t_1,t_1]^{p/2}}dx_1\cdots dx_{p/2}\int_{[-t_2,t_2]^{(q/2)-1}}dx_{(p/2)+1}\cdots dx_{-1+(p+q)/2}\\
\times&\prod_{j=1}^pI_{[0,1]}(x_0+b\sum\limits_{i=1}^jy_i)\prod_{j'=p+1}^{p+q}I_{[0,1]}(y_0+b\sum\limits_{i=p+1}^{j'}y_i)
\end{align*}
and
\begin{align*}
\tilde f_{II}^+(\pi,t_1,t_2)=&\int_{[0,1]^2}dy_0dx_0\int_{[-t_1,t_1]^{p/2}}dx_1\cdots dx_{p/2}\int_{[-t_2,t_2]^{(q/2)-1}}dx_{(p/2)+1}\cdots dx_{-1+(p+q)/2}\\
\times&\prod_{j=1}^pI_{[0,1]}(x_0+b\sum\limits_{i=1}^jy_i)\prod_{j'=p+1}^{p+q}I_{[0,1]}(y_0-b\sum\limits_{i=p+1}^{j'}y_i).
\end{align*}
\begin{remark}
$\tilde f_I^\pm(\pi,t_1,t_2)$ and $\tilde f_{II}^\pm(\pi,t_1,t_2)$
are same as the functions $f_I^\pm(\pi)$ and $f_{II}^\pm(\pi)$
defined in \cite{lsw} respectively,
only except that the domains of the integrals are different.\\
\end{remark}
Immediately we have that when $b=0$,
\begin{align}\label{eq: to compute integral of type I when b=0}
\tilde f_I^-(\pi,t_1,t_2)=\tilde
f_I^+(\pi,t_1,t_2)=\int_{[-t_1,t_1]^s}dx_1\cdots
dx_s\int_{[-t_2,t_2]^{-s+(p+q)/2}}dx_{s+1}\cdots
dx_{(p+q)/2}\cdot\delta(\sum\limits_{i=1}^py_i)
\end{align}

\begin{align}\label{eq: to compute integral of type II when b=0}
\tilde f_{II}^-(\pi,t_1,t_2)=\tilde
f_{II}^+(\pi,t_1,t_2)=\int_{[-t_1,t_1]^{p/2}}dx_1\cdots
dx_{p/2}\int_{[-t_2,t_2]^{(q/2)-1}}dx_{(p/2)+1}\cdots
dx_{-1+(p+q)/2}
\end{align}

\begin{lemma}\label{lemma: integrals for b=0}
Suppose $p+q$ is even and $\pi_1\in\mathcal{P}_2(p,q)$. If $\pi_1$
has $k$ crosses, then when $b=0$,
\begin{align}\label{eq: integral of type I when b=0}
\tilde f_I^-(\pi_1,t_1,t_2)=\tilde
f_I^+(\pi_1,t_1,t_2)=2^{\frac{p+q}{2}-1}\,t_1^{\frac{p+k}{2}-1}\,t_2^{\frac{q-k}{2}}.
\end{align}
Suppose $p,q$ are both even and $\pi_2\in\mathcal{P}_{2,4}(p,q)$.
Then when $b=0$,
\begin{align}\label{eq: integral of type II when b=0}
\tilde f_{II}^-(\pi_2,t_1,t_2)=\tilde
f_{II}^+(\pi_2,t_1,t_2)=2^{\frac{p+q}{2}-1}\,t_1^{\frac{p}{2}}\,
t_2^{\frac{q}{2}-1}.
\end{align}
\end{lemma}

\begin{proof}[Proof of Lemma \ref{lemma: integrals for b=0}]
Since $\pi_1$ has $k$ crosses, it has $k$ blocks which intersect
both $\{1,\ldots,p\}$ and $\{p+1,\ldots,p+q\}$. Therefore $\pi_1$
has $\dfrac{p-k}{2}$ blocks totally contained in $\{1,\ldots,p\}$
and $\dfrac{q-k}{2}$ blocks totally contained in
$\{p+1,\ldots,p+q\}$. By definition of the number $s$ we have that
\[
s=\dfrac{p-k}{2}+k=\dfrac{p+k}{2}.
\]
The variables $x_{s+1},\ldots,x_{(p+q)/2}$ correspond to the blocks
totally contained in $\{p+1,\ldots,p+q\}$ thus can take values
freely in $[-t_2,t_2]$. So their contribution to the integral is
$$(2t_2)^{((p+q)/2)-s}=(2t_2)^{(q-k)/2}.$$ Among $x_1,\ldots,x_s$, there are
$\dfrac{p-k}{2}$ of them corresponding to the blocks totally
contained in $\{1,\ldots,p\}$ and they can take value freely in
$[-t_1,t_1]$. So their contribution to the integral is
$$(2t_1)^{(p-k)/2}.$$ The other $k$ variables of $x_1,\ldots,x_s$
correspond to the $k$ crosses. They can only take value in
$[-t_1,t_1]$ but not $[-t_2,t_2]$ since $|t_1|\le|t_2|$. The
restriction $\delta(\sum\limits_{i=1}^py_i)$ is equivalent to the
fact that the sum of these $k$ variables is $0$. Thus this
restriction will take off one degree of freedom of these $k$
variables and their contribution to the integral is
$$(2t_1)^{k-1}.$$ The total integral should be the product of
contribution of all variables, which is
\[
(2t_2)^{(q-k)/2}\cdot(2t_1)^{(p-k)/2}\cdot(2t_1)^{k-1}=2^{(\frac{p+q}{2}-1)}t_1^{(\frac{p+k}{2}-1)}t_2^{(\frac{q-k}{2})}.
\]
Thus we proved (\ref{eq: integral of type I when b=0}). (\ref{eq:
integral of type II when b=0}) comes directly from (\ref{eq: to
compute integral of type II when b=0}).
\end{proof}
}


\section{Fluctuation of eigenvalues for matrix with bandwidth proportional to
$t$}\label{sec: fluctuation} The purpose of this section is to prove
Theorem \ref{thm:multi time fluctuation}.

\subsection{Covariance structure of $ \omega_p(t) $}\label{sec: covariance for fluctuation}
The asymptotic covariance structure  of $ \omega_p(t) $ is given by
Lemma \ref{lemma:covariance structure}.
\begin{lemma}\label{lemma:covariance structure}
Suppose all assumptions of Theorem \ref{thm:multi time fluctuation}
hold. If $t_1\le t_2$ are positive numbers in $(0,1/b)$ and $p$, $q$
are natural numbers no less than 2, then
\begin{multline*}
\lim\limits_{n\to\infty}\mathbb{E}[\omega_p(t_1)\omega_q(t_2)]\\
=\sum_{\pi\in \mathcal{\tilde P}_{2 }(p,q)}\left(\tilde
f^{-}_{I}(\pi,t_1,t_2)+\tilde
f^{+}_{I}(\pi,t_1,t_2)\right)+(\kappa-1)\sum_{\pi\in
\mathcal{P}_{2,4}(p,q)}\left(\tilde f^{-}_{II}(\pi,t_1,t_2)+\tilde
f^{+}_{II}(\pi,t_1,t_2)\right)
\end{multline*}
where  $ \omega_p(t_1) ,\omega_q(t_2) $ are defined in
\eqref{eq:omega p} of Section \ref{sec: fluctuation result}. The
integrals $\tilde f^{\pm}_{I}(\pi,t_1,t_2)$ and $\tilde
f^{\pm}_{II}(\pi,t_1,t_2)$ are defined in Appendix \ref{appendix}.
\end{lemma}

\begin{remark}\label{remark: covairance of Liu Sun Wang's model}
From Lemma \ref{lemma:covariance structure}, the $\sigma_p^2$ which
appeared in Theorem \ref{thm: fluctuation thm of Liu,Sun,Wang} is
\begin{align*}
\sigma_p^2=\sum_{\pi\in \mathcal{\tilde P}_{2 }(p,p)}\left(\tilde
f^{-}_{I}(\pi,1,1)+\tilde f^{+}_{I}(\pi,1,1)\right)\nonumber
+(\kappa-1)\sum_{\pi\in \mathcal{P}_{2,4}(p,p)}\left(\tilde
f^{-}_{II}(\pi,1,1)+\tilde f^{+}_{II}(\pi,1,1)\right)
\end{align*}
\end{remark}

\begin{proof}[Proof of Lemma \ref{lemma:covariance structure}]
By Lemma \ref{lemma:trace formula} we have
\begin{align*}
\text{tr}A_n(t_1)^p=(\frac{1}{\sqrt{b_n}})^p\sum\limits_{i=1}^n\sum\limits_{j_1,\ldots,j_p=-[b_nt_1]}^{[b_nt_1]}a_{j_1}\cdots a_{j_p}\prod_{l=1}^pI_{[1,n]}(i+\sum\limits_{r=1}^lj_r)\delta(\sum\limits_{r=1}^pj_r),\\
\text{tr}A_n(t_2)^q=(\frac{1}{\sqrt{b_n}})^q\sum\limits_{i=1}^n\sum\limits_{j_1,\ldots,j_q=-[b_nt_2]}^{[b_nt_2]}a_{j_1}\cdots
a_{j_q}\prod_{l=1}^qI_{[1,n]}(i+\sum\limits_{r=1}^lj_r)\delta(\sum\limits_{r=1}^qj_r).
\end{align*}
So
\begin{align}\label{eq: covariance of Omega}
\E[\omega_p(t_1)\omega_q(t_2)]=&(\frac{\sqrt{b_n}}{n})^2\E[(\text{tr}A_n(t_1)^p-\E[\text{tr}A_n(t_1)^p])(\text{tr}A_n(t_2)^q-\E[\text{tr}A_n(t_2)^q])]\nonumber\\
=&\frac{1}{n^2}\frac{1}{b_n^{\frac{p+q}{2}-1}}\sum\limits_{i_1,i_2=1}^n\sum\limits_{\mathbf{J},\mathbf{J'}}\E[(a_{\mathbf{J}}-\E a_{\mathbf{J}})(a_{\mathbf{J'}}-\E a_{\mathbf{J'}})]I(i_1,\mathbf{J})I(i_2,\mathbf{J'})\nonumber\\
=&\frac{1}{n^2}\frac{1}{b_n^{\frac{p+q}{2}-1}}\sum\limits_{i_1,i_2=1}^n\sum\limits_{\mathbf{J},\mathbf{J'}}[\E[a_{\mathbf{J}}a_{\mathbf{J'}}]-\E a_{\mathbf{J}}\E a_{\mathbf{J'}}]I(i_1,\mathbf{J})I(i_2,\mathbf{J'})\nonumber\\
\end{align}
where $\mathbf{J}:=(j_1,\ldots,j_p)$ which runs over
$\{\mathbf{J}\in\{\pm1,\ldots,\pm{[b_nt_1]}\}^p|j_1+\cdots+j_p=0\}$,
$\mathbf{J'}:=(j'_1,\ldots,j'_q)$ which runs over
$\{\mathbf{J'}\in\{\pm1,\ldots,\pm{[b_nt_2]}\}^q|j'_1+\cdots+j'_q=0\}$,
$a_{\mathbf{J}}:=a_{j_1}\cdots a_{j_p}$,
$a_{\mathbf{J'}}:=a_{j'_1}\cdots a_{j'_q}$,
$I(i_1,\mathbf{J}):=\prod_{l=1}^pI_{[1,n]}(i_1+\sum\limits_{r=1}^lj_r)$
and
$I(i_2,\mathbf{J'}):=\prod_{l=1}^qI_{[1,n]}(i_2+\sum\limits_{r=1}^lj'_r)$.
The components of $\mathbf{J}$ and $\mathbf{J'}$ do not take the
value of 0 since $a_0=0$.

Notice that $\E[\omega_p(t_1)\omega_q(t_2)]$ has the same expression
as $\E[\omega_p\omega_q]$ in Section 4 of \cite{lsw}, only except
that in the expression of $\E[\omega_p\omega_q]$ there $\mathbf{J}$
runs over $\{\{\pm1,\ldots,\pm{b_n}\}^p|j_1+\cdots+j_p=0\}$ and
$\mathbf{J'}$ runs over
$\{\{\pm1,\ldots,\pm{b_n}\}^q|j'_1+\cdots+j'_q=0\}$. But this
difference does not cause much problem. So we can use the same
method as in the proof of Theorem 4.1 of \cite{lsw} to compute the
limit of $\E[\omega_p(t_1)\omega_q(t_2)]$.

First recall some definitions given in Section \ref{sec: blanced
vector and cluster}. For a vector $\mathbf{v}=(a_1,\ldots,a_k)$, set
$S_{\mathbf{v}}=\{|a_1|,\ldots,|a_k|\}$. $S_{\mathbf{v}}$ is a set
of non-negative numbers with multiplicity. Two vectors
$\mathbf{v_1},\mathbf{v_2}$ are called correlated if
$S_{\mathbf{v_1}}$ and $S_{\mathbf{v_2}}$ have at least one common
element. If $a_i$ is a component of $\mathbf{v_1}$ and $|a_i|\in
S_{\mathbf{v_1}}\cap S_{\mathbf{v_2}}$, then $a_i$ is called a joint
point of $\mathbf{v_1}$ and $\mathbf{v_2}$. A vector is called
balanced if the sum of its components is 0.

So by independence of the random variables, a term
$\E[a_{\mathbf{J}}a_{\mathbf{J'}}]-\E a_{\mathbf{J}}\E
a_{\mathbf{J'}}$ in (\ref{eq: covariance of Omega}) is not 0 only
when each element of $S_{\mathbf{J}}\cup S_{\mathbf{J'}}$ has
multiplicity 2 or more.

Most of the following parts of this proof are taken from the proof
of Theorem 4.1 of \cite{lsw}.


We first  construct a map from the ordered correlated pair
$\mathbf{J}=(j_1,\ldots,j_p)$ and $\mathbf{J'}=(j'_1,\ldots,j'_q)$
as follows. Let $j_{u}\in \mathbf{J}$ be the first joint point
(whose subscript is the smallest) of the ordered correlated pair
$\mathbf{J}$ and $\mathbf{J'}$, and let $j'_v$ be the first element
in $\mathbf{J'}$ such that $|j_{u}|=|j'_{v}|$. If $j_{u}=-j'_{v}$,
we construct a vector $\mathbf{L}=(l_{1},\ldots,l_{p+q-2})$ such
that
\begin{align}l_{1}=j_{1},\ldots,&l_{u-1}=j_{u-1},
l_{u}=j'_{1},\ldots,l_{u+v-2}=j'_{v-1},\nonumber\\
&l_{u+v-1}=j'_{v+1}, \ldots,l_{u+q-2}=j'_{q},
l_{u+q-1}=j_{u+1},\ldots l_{p+q-2}=j_{p}.\nonumber
\end{align}
It is obvious that $$\sum\limits_{k=1}^{p+q-2}l_{k}=0,$$ so
$\mathbf{L}$ is balanced. If $j_{u}=j'_{v}$, then from $\mathbf{J}$
and $-\mathbf{J'}=(-j'_{1},\ldots,-j'_{q})$ we proceed in the way as
above. We call this process of constructing $\mathbf{L}$ from
$\mathbf{J}$ and $\mathbf{J'}$ a $reduction$ step and denote it by
$\mathbf{L}=\mathbf{J}{\bigvee}_{|j_u|}\mathbf{J'}$.

Notice that the reduction might cause the appearance of one number
with multiplicity 1 in $S_{\mathbf{L}}$, although each number in the
union of $S_{\mathbf{J}}$ and $S_{\mathbf{J'}}$ occurs at least
twice. If so, the resulting
 number with multiplicity 1 in $S_{\mathbf{L}}$ must be coincident with the joint point
 $j_{u}$.

 Next, suppose we have a balanced vector $\mathbf{L}$ of dimension ($p+q-2$). We shall estimate in how many different
 ways  it can be obtained from correlated pairs of dimensions $p$ and $q$. First, we have to choose
 some component
 $l_{u}$ in the first half of the vector, $1\leq u \leq p$ such that
 \be{\label{premagecondition1}\left| \sum\limits_{i=u}^{u+q-2}l_{i}\right|\neq |l_{j}|,\ \ j=1,\ldots,u-1.}\ee
 Set $\mathbf{J}=(j_{1},\ldots,j_{p})$ with
\be{j_{1}=l_{1},\ldots,j_{u-1}=l_{u-1},j_{u}=\sum\limits_{i=u}^{u+q-2}l_{i},j_{u+1}=l_{u+q-1},
\ldots,j_{p}=l_{p+q-2}. }\ee  We also  have to choose
 some component
 $l_{u+v-1}$, $1\leq v \leq q-1$ such that
 \be{\label{premagecondition2}\left| \sum\limits_{i=u}^{u+q-2}l_{i}\right|\neq |l_{j}|,\ \ j=u,\ldots,u+v-2}\ee whenever $v\geq 2$.
  Set $\mathbf{J'}=(j'_{1},\ldots,j'_{q})$ with
\be{j'_{1}=l_{u},\ldots,j'_{v-1}=l_{u+v-2},j'_{v}=-\sum\limits_{i=u}^{u+q-2}l_{i},j'_{v+1}=l_{u+v-1},
\ldots,j'_{p}=l_{u+q-2}.}\ee If $j_{u}$ is the joint point of the
constructed correlated pair $\mathbf{J}$ and $\mathbf{J'}$ and
$j'_v$ is the corresponding element in $\mathbf{J'}$, then the pair
$\{\mathbf{J},\mathbf{J'}\}$ or $\{\mathbf{J},-\mathbf{J'}\}$ is the
pre-image of $\mathbf{L}$. Note that since when $u=v=1$ the
conditions (\ref{premagecondition1}) and (\ref{premagecondition2})
are satisfied, the pre-image of $\mathbf{L}$ is always nonempty. A
simple estimation shows the following claim.\\

\centerline{ \textbf{Claim 1:} The number of pre-images of
$\mathbf{L}$ is at most $2pq$, a number not depending on $n$.}

\vspace{0.5cm} Remember that there is at most one element with
multiplicity 1 in $S_{\mathbf{L}}$. If there is one number with
multiplicity 1, then this number will be determined by others
because $\mathbf{L}$ is balanced. Consequently, the degree of
freedom for such terms is at most $\frac{p+q-2-1}{2}$, i.e., the
number of such $(\mathbf{J},\mathbf{J'})$ is
$O(b_n^{\frac{p+q-2-1}{2}})$. Therefore, the contribution of these
terms to $\E[\omega_p(t_1)\omega_q(t_2)]$ is $O(b^{-1/2}_{n})=o(1)$
because of the uniform boundedness of the moments of the entries
$a_j$ with order no more than $p+q$ and the coefficient
$\big(b_n^{\frac{p+q}{2}-1}\big)^{-1}$ in \eqref{eq: covariance of
Omega}. Now we suppose each number in $S_{\mathbf{L}}$ occurs at
least twice.

Thus $S_{\mathbf{L}}$ has at most $[\frac{p+q-2}{2}]$ distinct
elements. When elements of $S_{\mathbf{L}}$ are specified, there are
no more than
\[
2^{p+q-2}[\frac{p+q-2}{2}]^{p+q-2}
\]
ways to specify $l_1,\ldots,l_{p+q-2}$ and the above number does not
depend on $ n $. Again by the uniform boundedness of the moments of
the entries $a_j$ with order no more than $p+q$,
\[
\frac{1}{b_n^{\frac{p+q}{2}-1}}\sum\limits_{\mathbf{J},\mathbf{J'}}[\E[a_{\mathbf{J}}a_{\mathbf{J'}}]-\E
a_{\mathbf{J}}\E
a_{\mathbf{J'}}]=O(b_n^{[\frac{p+q-2}{2}]-\frac{p+q-2}{2}}).
\]

This implies that
\[
\lim\limits_{n\to\infty}\E[\omega_p(t_1)\omega_q(t_2)]=0
\]
when $p+q$ is odd. So we only need to consider the case that $p+q$
is even. If $S_{\mathbf{L}}$ has a term with multiplicity no less
than 3, then $S_{\mathbf{L}}$ has at most $[\frac{p+q-2-3}{2}+1]$
distinct numbers, thus the contribution of such terms to
$\E[\omega_p(t_1)\omega_q(t_2)]$ is again $o(1)$.

Therefore we only need to deal with the case that each element of
$S_{\mathbf{L}}$ has multiplicity 2, thus there exists
$\pi\in\mathcal{P}_2(p+q-2)$ (see Section \ref{sec:partition}) such
that if $s\sim_{\pi}w$ then $|l_{s}|=|l_{w}|$.

The condition
\[
l_1+\cdots+l_{p+q-2}=0
\]
implies that the main contribution to
$\E[\omega_p(t_1)\omega_q(t_2)]$ comes from the case that for each
pair $s\sim_{\pi}w$,
\[
 l_{s}=-l_{w}.
\]
That's because otherwise there exists $s_0,w_0$ such that
\[
l_{s_0}=l_{w_0}=-\frac{1}{2}\Big[\sum\limits_{t\in[p+q-2]\backslash\{s_0,w_0\}}l_t\Big]
\]
which implies that the value of $l_{s_0}$ and $l_{w_0}$ will be
determined by the value of  $l_t,t\not\in \{s_0,w_0\}$. Therefore
there is a loss of at least one degree of freedom, which makes the
contribution of such terms be $o(1)$. Then we have:

\vspace{0.5 cm}

\centerline{ \textbf{Claim 2:} The main contribution to (\ref{eq:
covariance of Omega}) comes from the case that $p+q$ is even and
$\mathbf{L}\in\Gamma_{1}(p+q-2)$.}

\vspace{0.5cm} Here $\Gamma_{1}(p+q-2)$ denotes a set of vectors in
$\R^{p+q-2}$: a vector is in $\Gamma_{1}(p+q-2)$ if and only if each
of its components has the same absolute value as exactly one other
component of the opposite sign. For $\mathbf{L}\in
\Gamma_{1}(p+q-2)$ the weight \be{\E(a_{\mathbf{J}}a_{\mathbf{J'}})-
\E[a_{\mathbf{J}}]\E[a_{\mathbf{J'}}]=\E[\prod^{p}_{s=1}a_{j_{s}}\prod^{q}_{t=1}a_{j'_{t}}]-
\E[\prod^{p}_{s=1}a_{j_{s}}]\E[\prod^{q}_{t=1}a_{j'_{t}}] }\ee
equals to $(\E[a_i^2])^{\frac{p+q}{2}}$ if $j_{u}$ is not coincident
with any component of $\mathbf{L}$; otherwise the weight is either
$\E[|a_{j_{u}}|^{4}](\E[a_i^2])^{\frac{p+q-4}{2}}$ or
$\E[|a_{j_{u}}|^{4}](\E[a_i^2])^{\frac{p+q-4}{2}}-(\E[a_i^2])^{\frac{p+q}{2}}$.

So far  we have found the terms leading to the main contribution.
Now we calculate the variance based on whether or not the fourth
moment appears. If the fourth moment doesn't appear, then
$j_{1},\ldots,j_{p}, j'_{1},\ldots,j'_{q}$ match in pairs. In the
abstract, by their subscripts  they can be treated as pair
partitions of $\{1, 2, \ldots, p, p+1, \ldots, p+q\}$ but with at
least one cross (i.e., in $\mathcal{P}_{2}(p,q)$). Thus, for every
$\pi \in \mathcal{P}_{2}(p,q)$, the summation can be a Riemann sum
and its limit becomes $\tilde f^{-}_{I}(\pi)$ ( it is $\tilde
f^{+}_{I}(\pi)$
 when the first
 coincident components in $\mathbf{J}$ and $\mathbf{J'}$   have the  same
 sign). If $\pi$ has only one cross, say $w_1\sim_\pi w_2$ with $w_1\le p<w_2$, then
$j_{w_1}=j'_{w_2}=0$ since $\mathbf{J},\mathbf{J'}$ are balanced.
Thus
\[
\E[a_{\mathbf{J}}a_{\mathbf{J'}}]-
\E[a_{\mathbf{J}}]\E[a_{\mathbf{J'}}]=0
\]
because $a_0=0$. If $\pi$ has two crosses, say $u_1\sim_\pi u_2$
with $u_1\le p<u_2$ and $v_1\sim_\pi v_2$ with $v_1\le p<v_2$, then
$|j_{u_1}|=|j'_{u_2}|=|j_{v_1}|=|j'_{v_2}|$ and the fourth moment
$\kappa$ must appear in $\E[a_{\mathbf{J}}a_{\mathbf{J'}}]-
\E[a_{\mathbf{J}}]\E[a_{\mathbf{J'}}]$. So we should only consider
the permutations in $\mathcal{P}_2(p,q)$ with at least 3 crosses for
the case that the fourth moment does not appear. In other words, we
should only consider permutations in $\mathcal{\tilde P}_2(p,q)$
(see Section \ref{sec:partition}) for this case.

On the other hand,   if the fourth moment does  appear,
 then
 $j_{1},\ldots,j_{p}, j'_{1},\ldots,j'_{q}$ match in pairs except that
 there exists a block with four elements. Therefore, from the balance of $\mathbf{J}$ and $\mathbf{J'}$ we know that the main contribution comes from such partitions:
$j_{1},\ldots,j_{p}$ and  $j'_{1},\ldots,j'_{q}$ both form pair
partitions; the block with four elements take respectively from a
pair of $j_{1},\ldots,j_{p}$ and  $j'_{1},\ldots,j'_{q}$. Otherwise,
the degree of freedom decreases by at least one.  So by their
subscripts $j_1,\ldots,j_p,j_1'\ldots,j_q'$ can be treated as
partitions in $\mathcal{P}_{2,4}(p,q)$. Similarly, for every $\pi
\in \mathcal{P}_{2,4}(p,q)$, the corresponding summation can be a
Riemann sum and its limit becomes $\tilde f^{-}_{II}(\pi)$ (it is
$\tilde f^{+}_{II}(\pi)$
 when the first
 coincident components in $\mathbf{J}$ and $\mathbf{J'}$   have the  same
 sign).

Noticing that the
 coincident components in $\mathbf{J}$ and $\mathbf{J'}$ may have the same or opposite
 sign, we conclude that
\begin{align}\label{eq: covariance for cite later}
 &\lim\limits_{n\to\infty}\E[\omega_{p}(t_1)\,\omega_{q}(t_2)]\nonumber\\
=&(\E[a_i^2])^{\frac{p+q}{2}}\Big[\sum_{\pi\in \mathcal{\tilde P}_{2
}(p,q)}\left(\tilde f^{-}_{I}(\pi,t_1,t_2) +\tilde
f^{+}_{I}(\pi,t_1,t_2)\right)\Big]\nonumber\\
+&\Big(\E[|a_{j_{u}}|^{4}](\E[a_i^2])^{\frac{p+q-4}{2}}-(\E[a_i^2])^{\frac{p+q}{2}}\Big)\Big[\sum_{\pi\in
\mathcal{P}_{2,4}(p,q)}\left(\tilde f^{-}_{II}(\pi,t_1,t_2)+\tilde
f^{+}_{II}(\pi,t_1,t_2)\right)\Big].
\end{align}
Since $\E[a_i^2]=1$ and $\E[a_i^4]=\kappa$ for all $i\ne0$, the
lemma is proved.

\end{proof}



\begin{proposition}\label{prop: covariance for the first model}
Suppose $b=0$. Let $p,q$ be natural numbers greater than 1 and
$0<t_1\le t_2$. Then
\begin{align*}
&\lim\limits_{n\to\infty}\E[\omega_{p}(t_1)\,\omega_{q}(t_2)]\nonumber\\
=&\begin{cases}\sum\limits_{k=3,5,\ldots,\min\{p,q\}}R_1(p,q,k)t_1^{\frac{p+k}{2}-1}t_2^{\frac{q-k}{2}}2^{\frac{p+q}{2}}&\text{if }p,q\text{ are both odd}\\
\sum\limits_{k=4,6,\ldots,\min\{p,q\}}R_1(p,q,k)t_1^{\frac{p+k}{2}-1}t_2^{\frac{q-k}{2}}2^{\frac{p+q}{2}}+(\kappa-1)R_2(p,q)t_1^{\frac{p}{2}}t_2^{\frac{q}{2}-1}2^{\frac{p+q}{2}}&\text{if
}p,q\text{ are both even}\\
0&\text{otherwise}
\end{cases}.
\end{align*}
\end{proposition}

\begin{proof}
For a finite set $\Phi$, we use $\big|\Phi\big|$ to denote its
cardinality.

When $p,q$ are both odd, $\mathcal{P}_{2,4}(p,q)=\emptyset$. So from
Lemma \ref{lemma:covariance structure}
\[
\lim\limits_{n\to\infty}\E[\omega_{p}(t_1)\,\omega_{q}(t_2)]=\sum_{\pi\in
\mathcal{\tilde P}_{2 }(p,q)}\left(\tilde
f^{-}_{I}(\pi,t_1,t_2)+\tilde f^{+}_{I}(\pi,t_1,t_2)\right).
\]
It's easy to see that for $\pi\in\mathcal{\tilde P}_{2 }(p,q)$, the
number of crosses of $\pi$ can only be a number in
$\{3,5,\ldots,\min\{p,q\}\}$ since $p,q$ are odd. Thus from lemma
\ref{lemma: integrals for b=0} we have
\[
\lim\limits_{n\to\infty}\E[\omega_{p}(t_1)\,\omega_{q}(t_2)]=\sum\limits_{k=3,5,\ldots,\min\{p,q\}}2^{\frac{p+q}{2}}t_1^{(\frac{p+k}{2}-1)}t_2^{(\frac{q-k}{2})}\cdot\bigg|\big\{\pi\in\mathcal{\tilde
P}_2(p,q)\big|\pi\text{ has }k\text{ crosses}\big\}\bigg|.
\]
Simple enumeration shows that $\bigg|\big\{\pi\in\mathcal{\tilde
P}_2(p,q)\big|\pi\text{ has }k\text{ crosses}\big\}\bigg|$ is
${p\choose k}{q\choose k}k!(p-k-1)!!(q-k-1)!!$. Recalling
$R_1(p,q,k)={p\choose k}{q\choose k}k!(p-k-1)!!(q-k-1)!!$, we proved
the case when $p,q$ are both odd.

When $p,q$ are both even, then the number of crosses of a
permutation $\pi\in\mathcal{\tilde P}_{2 }(p,q)$ can only be a
number in $\{4,6,\ldots,\min\{p,q\}\}$. So from lemma \ref{lemma:
integrals for b=0} we have
\begin{align*}
&\lim\limits_{n\to\infty}\E[\omega_{p}(t_1)\,\omega_{q}(t_2)]\\
=&\sum\limits_{k=4,6,\ldots,\min\{p,q\}}2^{\frac{p+q}{2}}t_1^{(\frac{p+k}{2}-1)}t_2^{(\frac{q-k}{2})}\cdot\bigg|\big\{\pi\in\mathcal{\tilde
P}_2(p,q)\big|\pi\text{ has }k\text{ crosses}\big\}\bigg|\\
+&(\kappa-1)2^{\frac{p+q}{2}}t_1^{(\frac{p}{2})}t_2^{(\frac{q}{2}-1)}\cdot\bigg|\mathcal{P}_{2,4}(p,q)\bigg|.
\end{align*}
By simple enumeration we have that
$\bigg|\mathcal{P}_{2,4}(p,q)\bigg|$ equals
${p\choose2}{q\choose2}(p-3)!!(q-3)!!=\frac{p\,q}{4}(p-1)!!(q-1)!!$.
Similarly as in the case of odd $p,q$ we have
$\bigg|\big\{\pi\in\mathcal{\tilde P}_2(p,q)\big|\pi\text{ has
}k\text{ crosses}\big\}\bigg|={p\choose k}{q\choose
k}k!(p-k-1)!!(q-k-1)!!$. Using the definitions of $R_1(p,q,k)$ and
$R_2(p,q,k)$ we complete the proof.
\end{proof}


\subsection{Multi-point fluctuation }\label{sec:Gaussian fluctuation proof}
In this section we prove Theorem \ref{thm:multi time fluctuation}.

\begin{proof}[Proof of Theorem \ref{thm:multi time fluctuation}]
The main idea of the proof is to show that the joint moments of
$(\omega_{p_1}(t_1),\ldots,\omega_{p_r}(t_r))$ converge to the joint
moments of a multivariate Gaussian distribution
$(X_{t_1},\ldots,X_{t_r})$.\\

\textbf{Step 1}\\

By Lemma \ref{lemma:trace formula} we have
\begin{multline*}
\E\big[\omega_{p_1}(t_1)\cdots\omega_{p_r}(t_r)\big]\\
=n^{-r}b_n^{-\frac{p_1+\cdots+p_r-r}{2}}\sum\limits_{i_1,\ldots,i_r=1}^n\sum\limits_{\mathbf{J_l},\ldots,\mathbf{J_r}}\E\big[(a_{\mathbf{J_l}}-\E a_{\mathbf{J_l}})\cdots(a_{\mathbf{J_r}}-\E a_{\mathbf{J_r}})\big]I(i_1,\mathbf{J_l})\cdots I(i_r,\mathbf{J_r})\\
\end{multline*}
where each $\mathbf{J_k}=(j_1(k),\ldots,j_{p_k}(k))$ which runs over

\begin{align}\label{mathbf{J_i}}
\Big\{(j_1(k),\ldots,j_{p_k}(k))\in\{\pm1,\ldots,\pm[b_nt_k]\}^{p_k}\Big|j_1(k)+\cdots+j_{p_k}(k)=0\Big\},
\end{align}
$a_{\mathbf{J_k}}=a_{j_1(k)}\cdots a_{j_{p_k}(k)}$ and
$I(i_k,\mathbf{J_k})=\prod_{s=1}^{p_k}I_{[1,n]}(i_k+j_1(k)+\cdots+j_s(k))$.\\

\textbf{Step 2}\\

We have the following observations.
\begin{enumerate}
\item The set \eqref{mathbf{J_i}} is a subset of
$$\Big\{(j_1(i),\ldots,j_{p_i}(i))\in\{\pm1,\ldots,\pm b_n\}^{p_i}\Big|j_1(i)+\cdots+j_{p_i}(i)=0\Big\}.$$

\item For fixed $\mathbf{J_l},\ldots,\mathbf{J_r}$, if there exists $1\le
i\le r$ such that $\mathbf{J_i}$ is correlated with none of the
others of $\mathbf{J_l},\ldots,\mathbf{J_r}$, then
$\E[(a_{\mathbf{J_l}}-\E a_{\mathbf{J_l}})\cdots(a_{\mathbf{J_r}}-\E
a_{\mathbf{J_r}})]=0$ because of the independence of
$a_1,a_2,\cdots$.


\item When $\mathbf{J_l},\ldots,\mathbf{J_r}$ vary,
$\E[(a_{\mathbf{J_l}}-\E a_{\mathbf{J_l}})\cdots(a_{\mathbf{J_r}}-\E
a_{\mathbf{J_r}})]$ is bounded because of the boundedness of the
moments of $a_1,a_2,\dotsm$. In other words, there exists $M>0$ such
that
$$|\E[(a_{\mathbf{J_l}}-\E a_{\mathbf{J_l}})\cdots(a_{\mathbf{J_r}}-\E
a_{\mathbf{J_r}})]|\le M$$ holds uniformly.
\end{enumerate}

\textbf{Step 3}\\

Consider $\{\mathbf{J_l},\ldots,\mathbf{J_r}\}$ where each
$\mathbf{J_i}$ is from (\ref{mathbf{J_i}}). It can always decompose
into several clusters (see Section \ref{sec: blanced vector and
cluster}), say $\mathbf{C_1},\ldots,\mathbf{C_d}$. If one of the
clusters has length 1, then from the independence of
$a_0,a_1,\ldots$ we have
$$\E[(a_{\mathbf{J_l}}-\E a_{\mathbf{J_l}})\cdots(a_{\mathbf{J_r}}-\E
a_{\mathbf{J_r}})]=0.$$ If a cluster $\mathbf{C_i}$ has two vectors,
say $\mathbf{C_i}=\{\mathbf{J_u},\mathbf{J_v}\}$, then by the same
argument as in the proof of Lemma \ref{lemma:covariance structure}
we know the number of ways we specify $\mathbf{J_u}$ and
$\mathbf{J_v}$ is $O(b_n^{(p_u+p_v-2)/2})$ as $n\to\infty$.
Therefore the contribution of $\mathbf{C_i}$ to
$\E[\omega_{p_1}(t_1)\cdots\omega_{p_r}(t_r)]$ is $O(1)$ as
$n\to\infty$. If a cluster $\mathbf{C_k}$ has more than two vectors,
say $\mathbf{C_k}=\{\mathbf{J_{k_1}},\ldots,\mathbf{J_{k_w}}\}$
$(w>2)$, then from Lemma \ref{lemma:cluster estimate}, the number of
ways we specify its vectors is
$o(b_n^{(p_{k_1}+\cdots+p_{k_w}-w)/2})$ as $n\to\infty$. So from the
above Observation 3, the contribution of $\mathbf{C_k}$ to
$\E[\omega_{p_1}(t_1)\cdots\omega_{p_r}(t_r)]$ is $o(1)$.

Therefore the contribution of $\{\mathbf{J_l},\ldots,\mathbf{J_r}\}$
to
$\E[\omega_{p_1}(t_1)\cdots\omega_{p_r}(t_r)]$ is not $o(1)$ only\\
when $\{\mathbf{J_l},\ldots,\mathbf{J_r}\}$ can decompose into
clusters of length 2. So $r$ is even.

This implies that
$\lim\limits_{n\to\infty}\E[\omega_{p_1}(t_1)\cdots\omega_{p_r}(t_r)]=0$
when $ r  $ is odd.

When $r$ is even,
\begin{align}\label{eq: limit of expectation of product of many
omega}
\lim\limits_{n\to\infty}\E\big[\omega_{p_1}(t_1)\cdots\omega_{p_r}(t_r)\big]
=&\lim\limits_{n\to\infty}\sum\limits_{\pi} \prod_{i=1}^{r/2}  \E\Big[   \omega_{p_{a(i)}}(t_{a(i)})\omega_{p_{b(i)}}(t_{b(i)})\Big]\nonumber\\
=&\sum\limits_{\pi} \prod_{i=1}^{r/2} \lim\limits_{n\to\infty}\E\Big[\omega_{p_{a(i)}}(t_{a(i)})\omega_{p_{b(i)}}(t_{b(i)})\Big] 
\end{align}
where \[ \pi=\Big\{\{a(1),b(1)\},\ldots,\{a(r/2),b(r/2)\}\Big\} \]
runs over $\mathcal{P}_2(r)$.\\

\textbf{Step 4}\\

Suppose \[ M_n=(m_{ij}^{(n)})_{i,j=1}^r \quad\textrm{ and }\quad
M=(m_{ij})_{i,j=1}^r \]  are $r\times r$ matrices with entries
$m_{ij}^{(n)}=\E[\omega_{p_i}(t_i)\omega_{p_j}(t_j)]$ and
$m_{ij}=\lim\limits_{n\to\infty}\E(\omega_{p_i}(t_i)\omega_{p_j}(t_j))$.

Since each $M_n$ is a covariance matrix, their limit $M$ is also a
covariance matrix of a centered multivariate Gaussian variable
$(X_{p_1}(t_1),\ldots,X_{p_r}(t_r))$.

From Wick's formula we have
\begin{align*}
\E[X_{p_1}(t_1)\cdots
X_{p_r}(t_r)]=\sum\limits_{\pi}\E[X_{p_{a(1)}}(t_{a(1)})X_{p_{b(1)}}(t_{b(1)})]\cdots\E[X_{p_{a(r/2)}}(t_{a(r/2)})X_{p_{a(r/2)}}(t_{b(r/2)})]
\end{align*}
where $\pi=\{\{a(1),b(1)\},\ldots,\{a(r/2),b(r/2)\}\}$ runs over
$\mathcal{P}_2(r)$. So
\begin{align*}
\lim\limits_{n\to\infty}\E[\omega_{p_1}(t_1)\cdots\omega_{p_r}(t_r)]=\E[X_{p_1}(t_1)\cdots
X_{p_r}(t_r)].
\end{align*}

In above equation $p_1,\ldots,p_r$ or $t_1,\ldots,t_r$ do not have
to be pairwise distinct. Thus by doing the same argument we can show
that all the joint moments of
$(\omega_{p_1}(t_1),\ldots,\omega_{p_r}(t_r))$ converges to the
corresponding joint moments of $(X_{p_1}(t_1),\ldots,X_{p_r}(t_r))$.
Therefore
\begin{align*}
(\omega_{p_1}(t_1),\ldots,\omega_{p_r}(t_r)) \text{ converges weakly
to } (X_{p_1}(t_1),\ldots,X_{p_r}(t_r)).
\end{align*}
Therefore there exists a family of centred Gaussian processes
$\{Z_p(t):t>0,p\ge2\}$ such that
$(Z_{p_1}(s_1),\ldots,Z_{p_k}(s_k))$ is distributed as
$(X_{p_1}(s_1),\ldots,X_{p_k}(s_k))$ for all $p_1,\ldots,p_k\in
\{2,3,\ldots\}$ and  $s_1<\ldots< s_k\in(0,1/b)$. Thus
\[
\lim\limits_{n\to\infty}\P(\omega_{p_1}(s_1)\le
a_1,\ldots,\omega_{p_k}(s_k)\le a_k)=\P(Z_{p_1}(s_1)\le
a_1,\ldots,Z_{p_k}(s_k)\le a_k)
\]
$\forall \{a_1,\ldots, a_k\}\subset\R$. Combining this with Lemma
\ref{lemma:covariance structure} and Proposition \ref{prop:
covariance for the first model}, we complete the proof.

\old { Moreover, we have the covariance structure of $\{Z_p(t)\}$:
for $0<t_1\le t_2$ and natural numbers $p>1$, $q>1$,
\begin{multline*}
\E[Z_p(t_1)Z_q(t_2)]
=\sum_{\pi\in \mathcal{\tilde P}_{2 }(p,q)}\left(\tilde
f^{-}_{I}(\pi,t_1,t_2)+\tilde
f^{+}_{I}(\pi,t_1,t_2)\right)+(\kappa-1)\sum_{\pi\in
\mathcal{P}_{2,4}(p,q)}\left(\tilde f^{-}_{II}(\pi,t_1,t_2)+\tilde
f^{+}_{II}(\pi,t_1,t_2)\right).
\end{multline*}
From Proposition \ref{prop: covariance for the first model} we see
that when $b=0$, $\E[Z_p(t_1)Z_q(t_2)]$ is exactly as shown in
Theorem \ref{thm:multi time fluctuation}.

The proof of Theorem \ref{thm:multi time fluctuation} is completed.}
\end{proof}

\begin{remark}
We can also use classic Central Limit Theorem to prove that if $b=0$
then $Z_2(t)$ is a Brownian motion.
\end{remark}

\section{Fluctuation of eigenvalues for matrix with Brownian motion
entries}\label{sec: last section: fluctuation for Brownian entries}

The purpose of this section is to prove Theorem \ref{thm:
fluctuation for Brownian entries}. Suppose all assumptions of
\ref{thm: fluctuation for Brownian entries} hold.

\subsection{Asymptotic
commutativity for random Toeplitz matrices with slowly growing
bandwidth}\label{sec: asymptotic commutatuvity}

\begin{lemma}\label{lemma: trace formula for trA^ptr(A+B)^q}
Suppose $0<t_1\le t_2$ and $p$, $q$ are natural numbers no less than
2. Set $u_j=a_j(t_1)$, $v_j=a_j(t_2)-a_j(t_1)$. We have
\begin{align*}
&\E[\omega_p(t_1)\omega_q(t_2)]\\
=&b_n^{-\frac{p+q}{2}+1}\sum\limits_{r=0}^q{q\choose
r}\sum\limits_{j_1,\ldots,j_p,\atop j_1',\ldots,
j_q'}\Bigg(\Big(\E[u_{j_1}\cdots u_{j_p}u_{j_i'}\cdots
u_{j_r'}v_{j_{r+1}'}\cdots v_{j_q'}]-\E[u_{j_1}\cdots
u_{j_p}]\E[u_{j_i'}\cdots u_{j_r'}v_{j_{r+1}'}\cdots
v_{j_q'}]\Big)\\
&\delta(j_1+\cdots+j_p)\delta(j_1'+\cdots+j_q')\Bigg)+o(1).
\end{align*}
\end{lemma}

\begin{remark}
Lemma \ref{lemma: trace formula for trA^ptr(A+B)^q} shows that when
evaluating $\E[\omega_p(t_1)\omega_q(t_2)]$ we can treat
\begin{align*}
\frac{1}{\sqrt{b_n}}\Big(a_{i-j}(t_1)\delta_{|i-j|\le
b_n}\Big)_{i,j=1}^n\quad\text{and}\quad
\frac{1}{\sqrt{b_n}}\Big((a_{i-j}(t_2)-a_{i-j}(t_1))\delta_{|i-j|\le
b_n}\Big)_{i,j=1}^n
\end{align*}
as commutative matrices.
\end{remark}

\begin{proof} Obviously $u_{i_1}$ and $v_{i_2}$ are independent for arbitrary $i_1$ and $i_2$. Set
\begin{align*}
U=B_n(t_1)=\frac{1}{\sqrt{b_n}}\Big(u_{i-j}\delta_{|i-j|\le
b_n}\Big)_{i,j=1}^n\quad\text{and}\quad
V=B_n(t_2)-B_n(t_1)=\frac{1}{\sqrt{b_n}}\Big(v_{i-j}\delta_{|i-j|\le
b_n}\Big)_{i,j=1}^n.
\end{align*}
By trace formula (Lemma \ref{lemma:trace formula}),
\begin{align*}
&\E[\omega_p(t_1)\omega_q(t_2)]\\
=&\frac{b_n}{n^2}\Big(\E[\text{tr}B_n(t_1)^p\text{tr}B_n(t_2)^q]-\E[\text{tr}B_n(t_1)^p]\E[\text{tr}B_n(t_2)^q]\Big)\\
=&\frac{b_n}{n^2}\Big(\E[\text{tr}U^p\text{tr}(U+V)^q]-\E[\text{tr}U^p]\E[\text{tr}(U+V)^q]\Big)\\
=&b_n^{-\frac{p+q}{2}+1} \sum\limits_{j_1,\ldots,j_p,\atop
j_1',\ldots, j_q'}\Big(\E[u_{j_1}\cdots
u_{j_p}(u_{j_1'}+v_{j_1'})\cdots
(u_{j_q'}+v_{j_q'})]-\E[u_{j_1}\cdots
u_{j_p}]\E[(u_{j_1'}+v_{j_1'})\cdots
(u_{j_q'}+v_{j_q'})]\Big)\\
\times&\delta(j_1+\cdots+j_p)\delta(j_1'+\cdots+j_q')\Big(\frac{1}{n}\sum\limits_{i=1}^n\prod_{l=1}^pI_{[1,n]}(i+\sum\limits_{k=1}^lj_k)\Big)\Big(\frac{1}{n}\sum\limits_{i'=1}^n\prod_{l'=1}^qI_{[1,n]}(i'+\sum\limits_{k=1}^{l'}j_k')\Big)
\end{align*}
where $j_1,\ldots,j_p,j_1',\ldots, j_q'$ all run over $[-b_n,b_n]$.
By writing the $(u_{j_1'}+v_{j_1'})\cdots (u_{j_q'}+v_{j_q'})$ in
the last equation as a sum of monomials we see it suffices to prove
that
\begin{align}\label{eq:eq_of_asymptotic_commutativity}
&b_n^{-\frac{p+q}{2}+1} \sum\limits_{j_1,\ldots,j_p,\atop
j_1',\ldots, j_q'}\Big(\E[u_{j_1}\cdots u_{j_p}u_{j_1'}\cdots
u_{j_r'}v_{j_{r+1}'}\cdots v_{j_q'}]-\E[u_{j_1}\cdots
u_{j_p}]\E[u_{j_1'}\cdots u_{j_r'}v_{j_{r+1}'}\cdots
v_{j_q'}]\Big)\nonumber\\
\times&\delta(j_1+\cdots+j_p)\delta(j_1'+\cdots+j_q')\Big(\frac{1}{n}\sum\limits_{i=1}^n\prod_{l=1}^pI_{[1,n]}(i+\sum\limits_{k=1}^lj_k)\Big)\Big(\frac{1}{n}\sum\limits_{i'=1}^n\prod_{l'=1}^qI_{[1,n]}(i'+\sum\limits_{k=1}^{l'}j_{\pi(k)}')\Big)\nonumber\\
=&b_n^{-\frac{p+q}{2}+1} \sum\limits_{j_1,\ldots,j_p,\atop
j_1',\ldots, j_q'}\Big(\E[u_{j_1}\cdots u_{j_p}u_{j_1'}\cdots
u_{j_r'}v_{j_{r+1}'}\cdots v_{j_q'}]-\E[u_{j_1}\cdots
u_{j_p}]\E[u_{j_1'}\cdots u_{j_r'}v_{j_{r+1}'}\cdots
v_{j_q'}]\Big)\nonumber\\
\times&\delta(j_1+\cdots+j_p)\delta(j_1'+\cdots+j_q')+o(1)
\end{align}
where $0\le r\le q$ and $\pi$ is a permutation in $S_q$.

Set $\mathbf{J}=(j_1,\ldots,j_p)$, $\mathbf{J'}=(j_1',\ldots,j_q')$.
The $reduction$ $\mathbf{L}$ (see the proof of Lemma
\ref{lemma:covariance structure}) of $\mathbf{J}$ and $\mathbf{J'}$
is a $p+q-2$-dimensional vector with components in
$\{-b_n,\ldots,b_n\}$. With the same argument as that in the proof
of Lemma \ref{lemma:covariance structure} we get the similar
conclusion as Claim 1 and Claim 2:
\begin{enumerate}
\item[(a)] For each $\mathbf{L}\in\mathbb{R}^{p+q-2}$, there are at most
$2pq$ pairs $(\mathbf{J},\mathbf{J'})$ whose reduction is
$\mathbf{L}$.
\item[(b)] The main contribution to the left hand side of
\eqref{eq:eq_of_asymptotic_commutativity} comes from the case that
each component of $\mathbf{L}$ has the same absolute value as
exactly one other component of the opposite sign.
\end{enumerate}
By (a) and (b), the main contribution to the left hand side and the
first term of the right hand side of
\eqref{eq:eq_of_asymptotic_commutativity} comes from
$(\mathbf{J},\mathbf{J'})\in A_n$ where $A_n$ is a subset of
$\{-b_n,\ldots,b_n\}^p\times\{-b_n,\ldots,b_n\}^q$ and the
cardinality of $A_n$ is $O(b_n^{\frac{p+q}{2}-1})$. Because of the
uniform boundedness of the moments of $\{u_i\}$ and $\{v_i\}$ with
order no more than $p+q$, the difference between the left hand side
and the first term of the right hand side of
\eqref{eq:eq_of_asymptotic_commutativity} is no more than
\begin{align}\label{difference}
C\cdot b_n^{-\frac{p+q}{2}+1}
\sum\limits_{(\mathbf{J},\mathbf{J'})\in A_n}
\Bigg|\Big(\frac{1}{n}\sum\limits_{i=1}^n\prod_{l=1}^pI_{[1,n]}(i+\sum\limits_{k=1}^lj_k)\Big)\Big(\frac{1}{n}\sum\limits_{i'=1}^n\prod_{l'=1}^qI_{[1,n]}(i'+\sum\limits_{k=1}^{l'}j_{\pi(k)}')\Big)-1\Bigg|
\end{align}
for some constant $C>0$. For $(\mathbf{J},\mathbf{J'})\in A_n$, the
components of $\mathbf{J}$ and $\mathbf{J'}$ are all in
$\{-b_n\ldots,b_n\}$, so
\begin{align*}
\Big|\frac{1}{n}\sum\limits_{i=1}^n\prod_{l=1}^pI_{[1,n]}(i+\sum\limits_{k=1}^lj_k)-1\Big|\le\frac{2pb_n}{n}&&\text{and}&&\Big|\frac{1}{n}\sum\limits_{i'=1}^n\prod_{l'=1}^qI_{[1,n]}(i'+\sum\limits_{k=1}^{l'}j_{\pi(k)}')-1\Big|\le\frac{2qb_n}{n}.
\end{align*}
Therefore \eqref{difference} is $O(\frac{b_n}{n})$ which is $o(1)$.
The proof is completed.
\end{proof}



\subsection{Covariance structure of $ \omega_p(t) $}\label{sec: covariance for Brownina
entries} In this section we will use the functions $\tilde
f_I^\pm(\pi,t_1,t_2)$ and
 $\tilde f_{II}^\pm(\pi,t_1,t_2)$  introduced in Appendix \ref{appendix}.

\begin{proposition}\label{prop: covariance for brownian entries}
If $0<t_1\le t_2$ and $p$, $q$ are natural numbers no less than 2,
then
\begin{align*}
\lim\limits_{n\to\infty}\E[\omega_p(t_1)\omega_q(t_2)]=&\begin{cases}\sum\limits_{r=2,4,\ldots,q}{q\choose
r}t_1^{\frac{p+r}{2}}R_3(p,r)(q-r-1)!!(t_2-t_1)^{\frac{q-r}{2}}2^{\frac{p+q}{2}}&\text{if
$p$ and $q$ are both even}\\\sum\limits_{r=3,5,\ldots,q}{q\choose
r}t_1^{\frac{p+r}{2}}R_4(p,r)(q-r-1)!!(t_2-t_1)^{\frac{q-r}{2}}2^{\frac{p+q}{2}}&\text{if
$p$ and $q$ are both odd}\\0&\text{otherwise}\end{cases}.
\end{align*}
\end{proposition}

\begin{proof}[Proof of Proposition \ref{prop: covariance for brownian entries}]

Set $u_j=a_j(t_1)$, $v_j=a_j(t_2)-a_j(t_1)$ and
\begin{align*}
U=A_n(t_1)=\frac{1}{\sqrt{b_n}}\Big(u_{i-j}\delta_{|i-j|\le
b_n}\Big)_{i,j=1}^n\quad\text{and}\quad
V=A_n(t_2)-A_n(t_1)=\frac{1}{\sqrt{b_n}}\Big(v_{i-j}\delta_{|i-j|\le
b_n}\Big)_{i,j=1}^n.
\end{align*}
Then $U$, $V$ are independent with $u_0=v_0=0$. For each $j\ne0$,
$u_j$ is a centered random variable with variance $t_1$ and $v_j$ is
a centered random variable with variance $t_2-t_1$.

From Lemma \ref{lemma: trace formula for trA^ptr(A+B)^q},
\begin{align*}
&\E[\omega_p(t_1)\omega_q(t_2)]\\
=&b_n^{-\frac{p+q}{2}+1}\sum\limits_{r=0}^q{q\choose
r}\sum\limits_{j_1,\ldots,j_p,\atop j_1',\ldots,
j_q'}\Bigg(\Big(\E[u_{j_1}\cdots u_{j_p}u_{j_i'}\cdots
u_{j_r'}v_{j_{r+1}'}\cdots v_{j_q'}]-\E[u_{j_1}\cdots
u_{j_p}]\E[u_{j_i'}\cdots u_{j_r'}v_{j_{r+1}'}\cdots
v_{j_q'}]\Big)\\
&\delta(j_1+\cdots+j_p)\delta(j_1'+\cdots+j_q')\Bigg)+o(1)
\end{align*}
where $j_1,\ldots,j_p,j_1',\ldots, j_q'$ all run over $[-b_n,b_n]$.

Set $\mathbf{J}=(j_1,\ldots,j_p)$, $\mathbf{J'}=(j_1',\ldots,j_q')$,
$\mathbf{J_1'}=(j_1',\ldots,j_r')$,
$\mathbf{J_2'}=(j_{r+1}',\ldots,j_q')$,
$u_{\mathbf{J}}=u_{j_1}\cdots u_{j_p}$,
$u_{\mathbf{J_1'}}=u_{j'_1}\cdots u_{j'_r}$ and
$v_{\mathbf{J_2'}}=v_{j'_{r+1}}\cdots v_{j'_q}$. Obviously
$\mathbf{J_1'}$ and $\mathbf{J_2'}$ are determined by $\mathbf{J'}$.
Then
\begin{align*}
\E[\omega_p(t_1)\omega_q(t_2)]&=b_n^{-\frac{p+q}{2}+1}\sum\limits_{r=0}^q{q\choose
r}\sum\limits_{\mathbf{J},\mathbf{J'}}\Bigg(\Big(\E[u_{\mathbf{J}}u_{\mathbf{J_1'}}v_{\mathbf{J_2'}}]-\E[u_{\mathbf{J}}]\E[u_{\mathbf{J_1'}}v_{\mathbf{J_2'}}]\Big)\Bigg)+o(1)\\
&=b_n^{-\frac{p+q}{2}+1}\sum\limits_{r=0}^q{q\choose
r}\sum\limits_{\mathbf{J},\mathbf{J'}}\Bigg(\Big(\E[u_{\mathbf{J}}u_{\mathbf{J_1'}}]-\E[u_{\mathbf{J}}]\E[u_{\mathbf{J_1'}}]\Big)\E[v_{\mathbf{J_2'}}]\Bigg)+o(1)
\end{align*}
where $\mathbf{J}$ runs over
$\{\mathbf{J}\in\{\pm1,\ldots,\pm{b_n}\}^p|j_1+\cdots+j_p=0\}$ and
$\mathbf{J'}$ runs over
$\{\mathbf{J'}\in\{\pm1,\ldots,\pm{b_n}\}^q|j'_1+\cdots+j'_q=0\}$.
The components of $\mathbf{J}$ and $\mathbf{J'}$ do not take the
value of 0 since $u_0=v_0=0$.

Now for fixed $r\in\{0,\ldots,q\}$ consider
\begin{align}\label{eq: r fixed}
b_n^{-\frac{p+q}{2}+1}\sum\limits_{\mathbf{J},\mathbf{J'}}\Bigg(\Big(\E[u_{\mathbf{J}}u_{\mathbf{J_1'}}]-\E[u_{\mathbf{J}}]\E[u_{\mathbf{J_1'}}]\Big)\E[v_{\mathbf{J_2'}}]\Bigg).
\end{align}

From the uniform bounds of the moments of the random variables,
there is a positive constant $C$ independent of $n$ such that
\[
\Big|\Big(\E[u_{\mathbf{J}}u_{\mathbf{J_1'}}]-\E[u_{\mathbf{J}}]\E[u_{\mathbf{J_1'}}]\Big)\E[v_{\mathbf{J_2'}}]\Big|<C
\]
for all $\mathbf{J}$ and $\mathbf{J'}$.

Recall that for a vector $\alpha=(\alpha_1,\ldots,\alpha_k)$, we set
$S_\alpha=\{|\alpha_1|,\ldots,|\alpha_k|\}$. From the independence
of the random variables $\{u_i\}$ and $\{v_i\}$, a term
$$\Big(\E[u_{\mathbf{J}}u_{\mathbf{J_1'}}]-\E[u_{\mathbf{J}}]\E[u_{\mathbf{J_1'}}]\Big)\E[v_{\mathbf{J_2'}}]$$
in \eqref{eq: r fixed} is not zero only when the following
constraints are satisfied:
\begin{description}
\item[Constraint 1:] $\mathbf{J}$ and $\mathbf{J_1'}$ are correlated, i.e., $S_{\mathbf{J}}$ and $S_{\mathbf{J_1'}}$ have common elements;
\item[Constraint 2:] each element in $S_{\mathbf{J}}\cup S_{\mathbf{J_1'}}$ has
multiplicity 2 or more;
\item[Constraint 3:] each element in $S_{\mathbf{J_2'}}$ has
multiplicity 2 or more.
\end{description}

Similarly as in the proof of Lemma \ref{lemma:covariance structure},
for each pair $(\mathbf{J},\mathbf{J'})$ corresponding to a non-zero
term we do reduction and get a $p+q-2$ dimensional balanced vector
$\mathbf{L}=\mathbf{J}{\bigvee}_{|j_u|}\mathbf{J'}$. From the
construction of $\mathbf{L}$ we notice that $\mathbf{J_2'}$ consists
of the last $q-r$ components of $\mathbf{L}$ since $\mathbf{J}$ and
$\mathbf{J_1'}$ are correlated. Different from the case in the proof
of Lemma \ref{lemma:covariance structure}, now the image of
reduction is not the whole set of
$\{\mathbf{L}\in\{\pm1,\ldots,\pm{b_n}\}^{p+q-2}|l_1+\cdots+l_{p+q-2}=0\}$
because of the above three constraints. Use $\Sigma$ to denote the
image of reductions of the pairs of balanced vectors
$(\mathbf{J},\mathbf{J'})$ satisfying the three constraints. From
the Claim 1 in the proof of Lemma \ref{lemma:covariance structure},
the pre-images of a given $\mathbf{L}\in\Sigma$ is no more than
$2pq$.

By exactly the same argument we used to get the Claim 2 in the proof
of Lemma \ref{lemma:covariance structure}, we get the following
observation.\\

\textbf{Observation} The main contribution to \eqref{eq: r fixed}
comes from the case that $p+q$ is even and each of $\mathbf{L}$'s
components has the same absolute value as exactly one other
component of the opposite sign.\\

 From the Constraint 3, if $j_s'$ is a component of
$\mathbf{J_2'}$, then there must be some other components
$j_{w_1}',\ldots,j_{w_\alpha}'$ of $\mathbf{J_2'}$ such that
$|j_s'|=|j_{w_1}'|=\cdots=|j_{w_\alpha}'|$. But all of
$j_s',j_{w_1}',\ldots,j_{w_\alpha}'$ appear in $\mathbf{L}$ since
$\mathbf{J_2'}$ consists of the last $q-r$ components of
$\mathbf{L}$. So from the above observation $j_s'$ has the same
absolute value as exactly one other component of $\mathbf{J_2'}$
which has an opposite sign as $j_s'$. In other words, the main
contribution to \eqref{eq: r fixed} comes from the case that $q-r$
is even and each of $\mathbf{J_2'}$'s components has the same
absolute value as exactly one other component of the opposite sign.
In this case $\mathbf{J_1'}$ and $\mathbf{J_2'}$ are both balanced.

So \eqref{eq: r fixed} equals
\begin{align}\label{eq: decompose of (24)}
\Bigg(b_n^{-\frac{p+r}{2}+1}\sum\limits_{\mathbf{J},\mathbf{J_1'}}\Big(\E[u_{\mathbf{J}}u_{\mathbf{J_1'}}]-\E[u_{\mathbf{J}}]\E[u_{\mathbf{J_1'}}]\Big)\Bigg)\Bigg(b_n^{-\frac{q-r}{2}}\sum\limits_{\mathbf{J_2'}}\E[v_{\mathbf{J_2'}}]\Bigg)+o(1)
\end{align}
where $\mathbf{J}$ runs over
$\{\mathbf{J}\in\{\pm1,\ldots,\pm{b_n}\}^p|\mathbf{J}\text{ is
balanced}\}$, $\mathbf{J_1'}$ runs over
$\{\mathbf{J_1'}\in\{\pm1,\ldots,\pm{b_n}\}^r|\mathbf{J_1'}\text{ is
balanced}\}$ and $\mathbf{J_2'}$ runs over
$\{\mathbf{J_2'}\in\{\pm1,\ldots,\pm{b_n}\}^{q-r}|\mathbf{J_2'}\text{
is balanced}\}$.

We can use exactly the same argument as in the proof of Lemma
\ref{lemma:covariance structure} to evaluate
$$\lim\limits_{n\to\infty}b_n^{-\frac{p+r}{2}+1}\sum\limits_{\mathbf{J},\mathbf{J_1'}}\Big(\E[u_{\mathbf{J}}u_{\mathbf{J_1'}}]-\E[u_{\mathbf{J}}]\E[u_{\mathbf{J_1'}}]\Big).$$
In fact, this limit is 0 when $p+r$ is odd. When $p+r$ is even, we
have the same formula as \eqref{eq: covariance for cite later}:
\begin{align*}
 &\lim\limits_{n\to\infty}b_n^{-\frac{p+r}{2}+1}\sum\limits_{\mathbf{J},\mathbf{J_1'}}\Big(\E[u_{\mathbf{J}}u_{\mathbf{J_1'}}]-\E[u_{\mathbf{J}}]\E[u_{\mathbf{J_1'}}]\Big)
\nonumber\\
=&(\E[u_i^2])^{\frac{p+r}{2}}\Big[\sum_{\pi\in \mathcal{\tilde P}_{2
}(p,r)}\left(\tilde f^{-}_{I}(\pi,1,1) +\tilde
f^{+}_{I}(\pi,1,1)\right)\Big]\nonumber\\
+&\Big((\E[u_i^2])^{\frac{p+r-4}{2}}\E[u_i^4]-(\E[u_i^2])^{\frac{p+r}{2}}\Big)\Big[\sum_{\pi\in
\mathcal{P}_{2,4}(p,r)}\left(\tilde f^{-}_{II}(\pi,1,1)+\tilde
f^{+}_{II}(\pi,1,1)\right)\Big].
\end{align*}

Recall that $\E[u_i^2]=t_1$ and $\E[u_i^4]=3t_1^2$ for all $i\ne0$
and that $$\tilde f^{-}_{I}(\pi,1,1)=\tilde
f^{+}_{I}(\pi,1,1)=\tilde f^{-}_{II}(\pi,1,1)+\tilde
f^{+}_{II}(\pi,1,1)=2^{\frac{p+r}{2}-1}$$ since $b_n=o(n)$ (see
Lemma \ref{lemma: integrals for b=0}). Therefore
\begin{align}\label{p,r term}
&\lim\limits_{n\to\infty}b_n^{-\frac{p+r}{2}+1}\sum\limits_{\mathbf{J},\mathbf{J_1'}}\Big(\E[u_{\mathbf{J}}u_{\mathbf{J_1'}}]-\E[u_{\mathbf{J}}]\E[u_{\mathbf{J_1'}}]\Big)\nonumber\\
=&\begin{cases}2^{\frac{p+r}{2}}t_1^{\frac{p+r}{2}}\Big(|\mathcal{\tilde
P}_{2 }(p,r)|+2|\mathcal{P}_{2,4}(p,r)|\Big)&\text{if $p+r$ is even}\\
0&\text{if $p+r$ is odd}\end{cases}
\end{align}
where $|\mathcal{\tilde P}_{2 }(p,r)|$ and
$|\mathcal{P}_{2,4}(p,r)|$ denote the cardinalities of
$\mathcal{\tilde P}_{2 }(p,r)$ and $\mathcal{P}_{2,4}(p,r)$
respectively.

To evaluate
$\lim\limits_{n\to\infty}b_n^{-\frac{q-r}{2}}\sum\limits_{\mathbf{J_2'}}\E[v_{\mathbf{J_2'}}]$,
consider a random Toeplitz matrix
\begin{align*}
M_n=\frac{1}{\sqrt{2b_n}}\Big(w_{i-j}\delta_{|i-j|\le
b_n}\Big)_{i,j=1}^n=\frac{1}{\sqrt{2b_n}}\Big(\frac{v_{i-j}}{\sqrt{t_2-t_1}}\,\delta_{|i-j|\le
b_n}\Big)_{i,j=1}^n
\end{align*}
where $w_i=\frac{v_{i-j}}{\sqrt{t_2-t_1}}$ are centered random
variables with variance 1. By Lemma \ref{lemma:trace formula},
\begin{align*}
\frac{1}{n}\E[\text{tr}M_n^{q-r}]=\frac{1}{n}\frac{1}{(2b_n)^{\frac{q-r}{2}}}\sum\limits_{i=1}^n\sum\limits_{j_1,\ldots,j_{q-r}=-b_n}^{b_n}\E[w_{j_1}\cdots
w_{j_{q-r}}]\prod_{s=1}^{q-r}I_{[1,n]}(i+\sum\limits_{z=1}^sj_z)\delta(\sum\limits_{z=1}^{q-r}j_z).
\end{align*}
Since $b_n=o(n)$, we have
$\frac{1}{n}\sum\limits_{i=1}^n\lim\limits_{n\to\infty}\prod_{s=1}^{q-r}I_{[1,n]}(i+\sum\limits_{z=1}^sj_z)=1$
for all $j_1,\ldots,j_{q-r}$. Thus from dominate convergence,
\begin{align*}
\frac{1}{n}\E[\text{tr}M_n^{q-r}]=&\frac{1}{(2b_n)^{\frac{q-r}{2}}}\sum\limits_{j_1,\ldots,j_{q-r}=-b_n}^{b_n}\E[w_{j_1}\cdots
w_{j_{q-r}}]\delta(\sum\limits_{z=1}^{q-r}j_z)+o(1)\\
=&\frac{1}{[2(t_2-t_1)]^{\frac{q-r}{2}}}\frac{1}{(b_n)^{\frac{q-r}{2}}}\sum\limits_{j_1,\ldots,j_{q-r}=-b_n}^{b_n}\E[v_{j_1}\cdots
v_{j_{q-r}}]\delta(\sum\limits_{z=1}^{q-r}j_z)+o(1)\\
=&\frac{1}{[2(t_2-t_1)]^{\frac{q-r}{2}}}\frac{1}{(b_n)^{\frac{q-r}{2}}}\sum\limits_{\mathbf{J_2'}}\E[v_{\mathbf{J_2'}}]+o(1)
\end{align*}
where $\mathbf{J_2'}$ runs over
$\{\mathbf{J_2'}\in\{\pm1,\ldots,\pm{b_n}\}^{q-r}|\mathbf{J_2'}\text{
is balanced}\}$. By Theorem 3.1 of \cite{lw} the empirical measure
of $M_n$ converges weakly to the standard Gaussian distribution. So
$\lim\limits_{n\to\infty}\frac{1}{n}\E[\text{tr}M_n^{q-r}]$ equals
$(q-r-1)!!$ if $q-r$ is even and $0$ if $q-r$ is odd. Thus
\begin{align}\label{q term}
\lim\limits_{n\to\infty}b_n^{-\frac{p+r}{2}}\sum\limits_{\mathbf{J_2'}}\E[v_{\mathbf{J_2'}}]=\begin{cases}0&\text{if
$q-r$ is
odd}\\(q-r-1)!!(t_2-t_1)^{\frac{q-r}{2}}2^{\frac{q-r}{2}}&\text{if
$q-r$ is even}\end{cases}.
\end{align}

From \eqref{eq: decompose of (24)}, \eqref{p,r term} and \eqref{q
term} we see that\eqref{eq: r fixed} is not $o(1)$ only when $p-r$
and $q-r$ are both even. Thus
\begin{align}\label{eq: omega}
&\lim\limits_{n\to\infty}\E[\omega_p(t_1)\omega_q(t_2)]\nonumber\\
=&\lim\limits_{n\to\infty}b_n^{-\frac{p+q}{2}+1}\sum\limits_{r=0}^q{q\choose
r}\sum\limits_{\mathbf{J},\mathbf{J'}}\Bigg(\Big(\E[u_{\mathbf{J}}u_{\mathbf{J_1'}}]-\E[u_{\mathbf{J}}]\E[u_{\mathbf{J_1'}}]\Big)\E[v_{\mathbf{J_2'}}]\Bigg)\nonumber\\
=&\begin{cases}\sum\limits_{r=2,4,\ldots,q}{q\choose
r}t_1^{\frac{p+r}{2}}2^{\frac{p+q}{2}}\Big(|\mathcal{\tilde P}_{2
}(p,r)|+2|\mathcal{P}_{2,4}(p,r)|\Big)(q-r-1)!!(t_2-t_1)^{\frac{q-r}{2}}&\text{if
$p$ and $q$ are both even}\\\sum\limits_{r=3,5,\ldots,q}{q\choose
r}t_1^{\frac{p+r}{2}}2^{\frac{p+q}{2}}|\mathcal{\tilde P}_{2
}(p,r)|(q-r-1)!!(t_2-t_1)^{\frac{q-r}{2}}&\text{if $p$ and $q$ are
both odd}\\0&\text{otherwise}\end{cases}
\end{align}
because $\mathcal{P}_{2,4}(p,r)=\emptyset$ when $p$ or $r$ is odd.
(We made the convention that $(-1)!!=1$.)

When $p$ and $r$ are even, as found in the proof of Proposition
\ref{prop: covariance for the first model},
\begin{align}\label{eq: enumeration 1}
\bigg|\mathcal{P}_{2,4}(p,r)\bigg|=\frac{p\,r}{4}(p-1)!!(r-1)!!.
\end{align}
Since $\mathcal{\tilde
P}_{2}(p,r)=\mathcal{P}_{2}(p+r)\backslash\bigg(\big\{\pi\in\mathcal{P}_{2}(p+r)|\pi\text{
has 0 cross}\big\}\cup\big\{\pi\in\mathcal{P}_{2}(p,r)|\pi\text{ has
2 crosses}\big\}\bigg)$,
\begin{align}\label{eq: enumeration 2}
|\mathcal{\tilde
P}_{2}(p,r)|=&|\mathcal{P}_{2}(p+r)|-|\mathcal{P}_{2}(p)||\mathcal{P}_{2}(r)|-|\{\pi\in\mathcal{P}_{2}(p,r)|\pi\text{ has 2 crosses}\}|\nonumber\\
=&(p+r-1)!!-(p-1)!!(r-1)!!-2{p\choose2}{r\choose2}(p-3)!!(r-3)!!\nonumber\\
=&(p+r-1)!!-(1+\frac{pr}{2})(p-1)!!(r-1)!!.
\end{align}

When $p$ and $r$ are odd, $\mathcal{\tilde
P}_{2}(p,r)=\mathcal{P}_{2}(p+r)\backslash\Big\{\pi\in\mathcal{P}_{2}(p+r)|\pi\text{
has 1 cross}\Big\}$, so
\begin{align}\label{eq: enumeration 3}
|\mathcal{\tilde
P}_{2}(p,r)|=|\mathcal{P}_{2}(p+r)|-|\{\pi\in\mathcal{P}_{2}(p,r)|\pi\text{
has 1 crosses}\}|=(p+r-1)!!-pr(p-2)!!(r-2)!!.
\end{align}
Plugging \eqref{eq: enumeration 1}, \eqref{eq: enumeration 2} and
\eqref{eq: enumeration 3} into \eqref{eq: omega} we finish the
proof.

\end{proof}

\subsection{Multi-point fluctuation }\label{multi-point fluctuation for Brownina entries}
We prove Theorem \ref{thm: fluctuation for Brownian entries} in this
Section.

\begin{proof}[Proof of Theorem \ref{thm: fluctuation for Brownian
entries}] The proof of Theorem \ref{thm: fluctuation for Brownian
entries} contains 4 steps. But only the first two steps are
different from the corresponding steps of the proof of Theorem
\ref{thm:multi time fluctuation}.

\textbf{Step 1}

By Lemma \ref{lemma:trace formula} we have that
\begin{align*}
&\E\big[\omega_{p_1}(t_1)\cdots\omega_{p_r}(t_r)\big]\\
=&\Big(\frac{\sqrt{b_n}}{n}\Big)^r\E\Big[\Big(\text{tr}A_n(t_1)^{p_1}-\E[\text{tr}A_n(t_1)^{p_1}]\Big)\cdots\Big(\text{tr}A_n(t_r)^{p_r}-\E[\text{tr}A_n(t_r)^{p_r}]\Big)\Big]\\
=&\Big(\frac{\sqrt{b_n}}{n}\Big)^r\Big(\frac{1}{\sqrt{b_n}}\Big)^{p_1+\cdots+p_r}\sum\limits_{i_1,\ldots,i_r=1}^n\sum\limits_{\mathbf{J_l},\ldots,\mathbf{J_r}}\E\big[(a_{\mathbf{J_l}}(t_1)-\E[a_{\mathbf{J_l}}(t_1)])\cdots(a_{\mathbf{J_r}}(t_r)-\E[a_{\mathbf{J_r}}(t_r)])\big]I(i_1,\mathbf{J_l})\cdots
I(i_r,\mathbf{J_r})
\end{align*}
where each $\mathbf{J_k}=(j_1(k),\ldots,j_{p_k}(k))$,
$a_{\mathbf{J_k}}(t_k)=a_{j_1(k)}(t_k)\cdots a_{j_{p_k}(k)}(t_k)$,
$I(i_k,\mathbf{J_k})=\prod\limits_{s=1}^{p_k}I_{[1,n]}(i_k+\sum\limits_{q=1}^sj_q(k))$
and $\mathbf{J_k}$ runs over

\begin{align*}
\Big\{(j_1(k),\ldots,j_{p_k}(k))\in\{\pm1,\ldots,\pm
b_n\}^{p_k}\Big|j_1(k)+\cdots+j_{p_k}(k)=0\Big\}.
\end{align*}

\textbf{Step 2}

We have the following the observations.
\begin{enumerate}
\item For fixed $\mathbf{J_l},\ldots,\mathbf{J_r}$, if there exists $1\le
i\le r$ such that $\mathbf{J_i}$ is not correlated with any other
one of $\mathbf{J_l},\ldots,\mathbf{J_r}$, then
$\E[(a_{\mathbf{J_l}}(t_1)-\E
[a_{\mathbf{J_l}}(t_1)])\cdots(a_{\mathbf{J_r}}(t_r)-\E
[a_{\mathbf{J_r}}(t_r)])]$ must be 0 because of the independence of
$a_1,a_2,\cdots$.


\item When $\mathbf{J_l},\ldots,\mathbf{J_r}$ vary,
$\E[(a_{\mathbf{J_l}}(t_1)-\E[a_{\mathbf{J_l}}(t_1)])\cdots(a_{\mathbf{J_r}}(t_r)-\E[a_{\mathbf{J_r}}(t_r)])]$
is bounded because of the boundedness of the moments of
$a_1,a_2,\dotsm$. So there is some $M>0$ such that
$$|\E[(a_{\mathbf{J_l}}(t_1)-\E[a_{\mathbf{J_l}}(t_1)])\cdots(a_{\mathbf{J_r}}(t_r)-\E[a_{\mathbf{J_r}}(t_r)])]|\le M$$ holds uniformly.
\end{enumerate}

Then by applying the same argument in Step 3 and Step 4 of the proof
of Theorem \ref{thm:multi time fluctuation}, we finish the proof of
Theorem \ref{thm: fluctuation for Brownian entries}.

\end{proof}

 \begin{appendix}
   \section{Some integrals in the proofs Theorem \ref{thm:multi time fluctuation} and \ref{thm: fluctuation for Brownian
   entries}}\label{appendix}

To prove  Theorem \ref{thm:multi time fluctuation} and Theorem
\ref{thm: fluctuation for Brownian entries}, we need to define four
functions $\tilde f_I^+(\pi,t_1,t_2)$, $\tilde f_I^-(\pi,t_1,t_2)$,
$\tilde f_{II}^+(\pi,t_1,t_2)$ and $\tilde f_{II}^-(\pi,t_1,t_2)$.

Suppose $b\in[0,1]$ and $t_1\le t_2$ are numbers in $(0,1/b)$. (We
define $(0,1/0)$ to be $(0,\infty)$.)

Suppose $p+q$ is even. For
$\pi=\{V_1,\ldots,V_{(p+q)/2}\}\in\mathcal{P}_2(p,q)$, set
\begin{align*}
\epsilon_{\pi}(i)=\begin{cases}1&\text{if $i$ is the smallest number
of }\pi^{-1}(\pi(i))\\-1&\text{otherwise}\end{cases}
\end{align*}
where $\pi(i)$ is defined in Section \ref{sec:partition}.

We construct a relation between two groups of unknowns
$y_1,\ldots,y_{p+q}$ and $x_1,\ldots,x_{\frac{p+q}{2}}$ as
$$\epsilon_{\pi}(i)y_i=\epsilon_{\pi}(j)y_j=x_{\pi(i)}$$ whenever
$i\sim_{\pi}j$. Without loss of generality suppose
$V_i\cap\{1,\ldots,p\}\ne\emptyset$ if and only if $i\le s$. For
$x_0,y_0\in[0,1]$ and $x_1,\ldots,x_{\frac{p+q}{2}}\in[-1,1]$, we
define
\begin{align*}
\tilde f_I^-(\pi,t_1,t_2)=&\int_{[0,1]^2}dy_0dx_0\int_{[-t_1,t_1]^s}dx_1\cdots dx_s\int_{[-t_2,t_2]^{-s+(p+q)/2}}dx_{s+1}\cdots dx_{(p+q)/2}\\
\times&\delta(\sum\limits_{i=1}^py_i)\prod_{j=1}^pI_{[0,1]}(x_0+b\sum\limits_{i=1}^jy_i)\prod_{j'=p+1}^{p+q}I_{[0,1]}(y_0+b\sum\limits_{i=p+1}^{j'}y_i)
\end{align*}
and
\begin{align*}
\tilde f_I^+(\pi,t_1,t_2)=&\int_{[0,1]^2}dy_0dx_0\int_{[-t_1,t_1]^s}dx_1\cdots dx_s\int_{[-t_2,t_2]^{-s+(p+q)/2}}dx_{s+1}\cdots dx_{(p+q)/2}\\
\times&\delta(\sum\limits_{i=1}^py_i)\prod_{j=1}^pI_{[0,1]}(x_0+b\sum\limits_{i=1}^jy_i)\prod_{j'=p+1}^{p+q}I_{[0,1]}(y_0-b\sum\limits_{i=p+1}^{j'}y_i)
\end{align*}
where $\delta$ is the Dirac function and $I_{[0,1]}$ is the
indicator function.

For
$\pi=\{V_1,\ldots,V_{\frac{p+q}{2}-1}\}\in\mathcal{P}_{2,4}(p,q)$
(denoting the block with four elements by $V_i$), we set for
$\pi(k)\ne i$
\begin{align*}
\tau_{\pi}(k)=\begin{cases}1& \text{if $k$ is the smallest number of
}\pi^{-1}(\pi(k));\\-1&\text{otherwise},\end{cases}
\end{align*}
while for $\pi(k)=i$
\begin{align*}
\tau_{\pi}(k)=\begin{cases}1& \text{if $k$ is the smallest or
largest number of
}\pi^{-1}(\pi(k));\\-1&\text{otherwise}.\end{cases}
\end{align*}
For such a $\pi$, we define a relation between two groups of
unknowns $y_1,\ldots,y_{p+q}$ and $x_1,\ldots,x_{\frac{p+q}{2}-1}$
as following: $$\tau_{\pi}(u)y_u=\tau_{\pi}(v)y_v=x_{\pi(u)}$$
whenever $u\sim_{\pi}v$.
For $x_0,y_0\in[0,1]$ and $x_1,\ldots,x_{\frac{p+q}{2}-1}\in[-1,1]$,
we define
\begin{align*}
\tilde f_{II}^-(\pi,t_1,t_2)=&\int_{[0,1]^2}dy_0dx_0\int_{[-t_1,t_1]^{p/2}}dx_1\cdots dx_{p/2}\int_{[-t_2,t_2]^{(q/2)-1}}dx_{(p/2)+1}\cdots dx_{-1+(p+q)/2}\\
\times&\prod_{j=1}^pI_{[0,1]}(x_0+b\sum\limits_{i=1}^jy_i)\prod_{j'=p+1}^{p+q}I_{[0,1]}(y_0+b\sum\limits_{i=p+1}^{j'}y_i)
\end{align*}
and
\begin{align*}
\tilde f_{II}^+(\pi,t_1,t_2)=&\int_{[0,1]^2}dy_0dx_0\int_{[-t_1,t_1]^{p/2}}dx_1\cdots dx_{p/2}\int_{[-t_2,t_2]^{(q/2)-1}}dx_{(p/2)+1}\cdots dx_{-1+(p+q)/2}\\
\times&\prod_{j=1}^pI_{[0,1]}(x_0+b\sum\limits_{i=1}^jy_i)\prod_{j'=p+1}^{p+q}I_{[0,1]}(y_0-b\sum\limits_{i=p+1}^{j'}y_i).
\end{align*}
\begin{remark}
$\tilde f_I^\pm(\pi,t_1,t_2)$ and $\tilde f_{II}^\pm(\pi,t_1,t_2)$
are same as the functions $f_I^\pm(\pi)$ and $f_{II}^\pm(\pi)$
defined in \cite{lsw} respectively,
only except that the domains of the integrals are different.\\
\end{remark}
Immediately we have that when $b=0$,
\begin{align}\label{eq: to compute integral of type I when b=0}
\tilde f_I^-(\pi,t_1,t_2)=\tilde
f_I^+(\pi,t_1,t_2)=\int_{[-t_1,t_1]^s}dx_1\cdots
dx_s\int_{[-t_2,t_2]^{-s+(p+q)/2}}dx_{s+1}\cdots
dx_{(p+q)/2}\cdot\delta(\sum\limits_{i=1}^py_i),
\end{align}

\begin{align}\label{eq: to compute integral of type II when b=0}
\tilde f_{II}^-(\pi,t_1,t_2)=\tilde
f_{II}^+(\pi,t_1,t_2)=\int_{[-t_1,t_1]^{p/2}}dx_1\cdots
dx_{p/2}\int_{[-t_2,t_2]^{(q/2)-1}}dx_{(p/2)+1}\cdots
dx_{-1+(p+q)/2}.
\end{align}

\begin{lemma}\label{lemma: integrals for b=0}
Suppose $p+q$ is even and $\pi_1\in\mathcal{P}_2(p,q)$. If $\pi_1$
has $k$ crosses, then when $b=0$,
\begin{align}\label{eq: integral of type I when b=0}
\tilde f_I^-(\pi_1,t_1,t_2)=\tilde
f_I^+(\pi_1,t_1,t_2)=2^{\frac{p+q}{2}-1}\,t_1^{\frac{p+k}{2}-1}\,t_2^{\frac{q-k}{2}}.
\end{align}
Suppose $p,q$ are both even and $\pi_2\in\mathcal{P}_{2,4}(p,q)$.
Then when $b=0$,
\begin{align}\label{eq: integral of type II when b=0}
\tilde f_{II}^-(\pi_2,t_1,t_2)=\tilde
f_{II}^+(\pi_2,t_1,t_2)=2^{\frac{p+q}{2}-1}\,t_1^{\frac{p}{2}}\,
t_2^{\frac{q}{2}-1}.
\end{align}
\end{lemma}

\begin{proof}[Proof of Lemma \ref{lemma: integrals for b=0}]
Since $\pi_1$ has $k$ crosses, it has $k$ blocks which intersect
both $\{1,\ldots,p\}$ and $\{p+1,\ldots,p+q\}$. Therefore $\pi_1$
has $\dfrac{p-k}{2}$ blocks totally contained in $\{1,\ldots,p\}$
and $\dfrac{q-k}{2}$ blocks totally contained in
$\{p+1,\ldots,p+q\}$. By definition of the number $s$ we have that
\[
s=\dfrac{p-k}{2}+k=\dfrac{p+k}{2}.
\]
The variables $x_{s+1},\ldots,x_{(p+q)/2}$ correspond to the blocks
totally contained in $\{p+1,\ldots,p+q\}$ thus can take values
freely in $[-t_2,t_2]$. So their contribution to the integral is
$$(2t_2)^{((p+q)/2)-s}=(2t_2)^{(q-k)/2}.$$ Among $x_1,\ldots,x_s$, there are
$\dfrac{p-k}{2}$ of them corresponding to the blocks totally
contained in $\{1,\ldots,p\}$ and they can take value freely in
$[-t_1,t_1]$. So their contribution to the integral is
$$(2t_1)^{(p-k)/2}.$$ The other $k$ variables of $x_1,\ldots,x_s$
correspond to the $k$ crosses. They can only take value in
$[-t_1,t_1]$ but not $[-t_2,t_2]$ since $|t_1|\le|t_2|$. The
restriction $\delta(\sum\limits_{i=1}^py_i)$ is equivalent to the
fact that the sum of these $k$ variables is $0$. Thus this
restriction will take off one degree of freedom of these $k$
variables and their contribution to the integral is
$$(2t_1)^{k-1}.$$ The total integral should be the product of
contribution of all variables, which is
\[
(2t_2)^{(q-k)/2}\cdot(2t_1)^{(p-k)/2}\cdot(2t_1)^{k-1}=2^{(\frac{p+q}{2}-1)}t_1^{(\frac{p+k}{2}-1)}t_2^{(\frac{q-k}{2})}.
\]
Thus we proved (\ref{eq: integral of type I when b=0}). (\ref{eq:
integral of type II when b=0}) comes directly from (\ref{eq: to
compute integral of type II when b=0}).
\end{proof}

  \end{appendix}

\textbf{Acknowledgement}\\
It is a pleasure to thank Mark Adler and Dangzheng Liu for reading
an early version of this paper and for their suggestions. We also
thank the anonymous reviewer for the helpful comments.

\Addresses

\end{document}